\pdfoutput=1
\documentclass[a4paper,11pt]{amsart}

\usepackage{latexsym, amssymb,amsthm,amsmath,color}
\usepackage{xspace}
\usepackage{mathtools}
\usepackage[all]{xy}
\usepackage[abs]{overpic}
\usepackage{hyperref}
\newcommand{\cc}{\ensuremath{\mathbb{C}}\xspace}
\newcommand{\nn}{\ensuremath{\mathbb{N}}\xspace}
\newcommand{\rr}{\ensuremath{\mathbb{R}}\xspace}

\newcommand{\abs}[1]{\left| #1 \right|}

\newcommand{\set}[1]{\lbrace #1 \rbrace}
\newcommand{\comp}{\hspace{0.1 mm} \circ \hspace{0.1 mm}}
\newcommand{\rest}[2]{\left.{#1}\right|_{#2}}
\renewcommand{\bar}{\overline}
\renewcommand{\tilde}{\widetilde}

\newcommand{\inp}[1]{\left\langle #1 \right\rangle}
\newcommand{\pd}[2]{\frac{\partial #1}{\partial #2}}
\DeclareMathOperator{\Res}{Res}

\newcommand{\oo}{\mathcal{O}}

\newcommand{\elli}{\mathcal{A}_{\abs{D}}}

\DeclareMathOperator{\Hess}{Hess}

\DeclarePairedDelimiter\floor{\lfloor}{\rfloor}

\renewcommand{\phi}{\varphi}
\usepackage{thmtools}
\declaretheorem[style=definition,qed=$\diamondsuit$]{definition}
\declaretheorem[style=definition,qed=$\triangle$,sibling=definition]{example}
\declaretheorem[style=plain,sibling=definition]{theorem}
\declaretheorem[style=plain,sibling=definition]{lemma}
\declaretheorem[style=plain,sibling=definition]{proposition}
\declaretheorem[style=plain,sibling=definition]{corollary}
\declaretheorem[style=definition,qed=$\diamondsuit$,sibling=example]{claim}
\declaretheorem[style=definition,qed=$\diamondsuit$,sibling=claim]{remark}

\newtheorem*{claim*}{Claim}
\numberwithin{theoremalpha}{section}

\numberwithin{equation}{section}
\numberwithin{definition}{section}
\numberwithin{theorem}{section}
\numberwithin{proposition}{section}
\numberwithin{lemma}{section}
\numberwithin{example}{section}
\numberwithin{remark}{section}
\numberwithin{corollary}{section}
\newcommand{\Ad}{\mathcal{A}_{\abs{D}}}
\newcommand{\cplx}{\mathcal{A}_D}

\definecolor{ao(english)}{rgb}{0.0, 0.5, 0.0}

\newcommand{\gcss}{generalized complex structures}
\newcommand{\gcs}{generalized complex structure}
\newtheoremstyle{named}{}{}{\itshape}{}{\bfseries}{.}{.5em}{\thmnote{#3}#1}
\theoremstyle{named}
\newtheorem*{namedtheorem}{}
\begin{document}
\title{Self-crossing stable generalized complex structures}
\author{Gil R. Cavalcanti}
\address{Department of Mathematics, Utrecht University, 3508 TA Utrecht, The Netherlands}
\email{g.r.cavalcanti@uu.nl}
\author{Ralph L.\ Klaasse}
\address{D\'epartement de Math\'ematique, Universit\'e libre de Bruxelles, Brussels 1050, Belgium}
\email{r.l.klaasse@gmail.com}
\author{Aldo Witte}
\address{Department of Mathematics, Utrecht University, 3508 TA Utrecht, The Netherlands}
\email{g.a.witte@uu.nl}
\begin{abstract} 
We extend the notion of (smooth) stable generalized complex structures to allow for an anticanonical section with normal self-crossing singularities. This weakening not only allows for a number of natural examples in higher dimensions but also sheds some light into the smooth case in dimension four. We show that in four dimensions there is a natural connected sum operation for these structures as well as a smoothing operation which changes a self-crossing stable generalized complex structure into a smooth stable generalized complex structure on the same manifold. This allows us to construct large families of stable generalized complex manifolds.
\end{abstract}
\maketitle
\tableofcontents

\section{Introduction}

Generalized complex structures, introduced by Hitchin \cite{H03} and Gualtieri \cite{Gua07}, are a simultaneous generalization of complex and symplectic structures. In fact, infinitesimally a generalized complex structure is equivalent to the product of a complex and a symplectic vector space. The number of complex directions a generalized complex structure has at a point is the {\it type} of the structure, which is an upper semi-continuous function. Points where the type vanishes form the {\it symplectic locus} while points of maximal type are of {\it complex type}. At a region where the type is constant and equal to, say, $k$, the structure is locally the product of $\mathbb{C}^k$ with its complex structure and $\mathbb{R}^{2(n-k)}$ with the standard symplectic structure \cite{Gua07}. Yet one of the striking features of generalized complex structures is that the type does not have to be locally constant, and complex and symplectic points may coexist in a connected manifold.

The anticanonical bundle of every generalized complex manifold comes equipped with a natural section which is nonzero precisely at the symplectic locus. The simplest type-changing phenomenon in generalized complex geometry happens in {\it smooth stable generalized complex structures} \cite{CG17,Goto16}. For those, the anticanonical section is transverse to the zero section. Despite of only displaying the simplest type-change behaviour, smooth stable generalized complex manifolds have a rich geometry and there are several interesting examples, especially in four-dimensions. Known examples on four-manifolds include that
\begin{enumerate}
\item $n\# \mathbb{C} P^2 \# m\overline{\mathbb{C} P^2}$ admits a smooth stable generalized complex structure if and only if it admits an almost-complex structure, that is if and only if $n$ is odd \cite{CG09};
\item $n\# \mathbb{C} P^2 \# m\overline{\mathbb{C} P^2} \#(S^1\times S^3)$ admits a smooth stable generalized complex structure if and only if it admits an almost-complex structure, that is if and only if $n$ is even \cite{Tor12};
\item $n\# (S^2 \times S^2)$ admits a smooth stable generalized complex structure if and only if it admits an almost-complex structure, that is if and only if $n$ is odd \cite{Tor12};
\item $n\# (S^2 \times S^2) \#(S^1\times S^3)$ admits a smooth stable generalized complex structure if and only if it admits an almost-complex structure, that is if and only if $n$ is even \cite{Tor12};
\item Elliptic surfaces admit non symplectic, smooth stable generalized complex structures \cite{GH16}.
\end{enumerate}
The families (1)~to (4)~only admit symplectic structures if $n = 1$ and only admit complex structures if $n=1$, or, in families (2)~and (4), if $n=0$.

Besides the existence of many examples, the geometry of these four-dimensional generalized complex structures is very reminiscent of  symplectic Landau--Ginzburg models from physics \cite{2000hep.th.5247H}. Precisely, our examples present themselves in a way that seems to be closely related to the setup used by Seidel in his study of the Fukaya category on symplectic manifolds with Lefschetz fibrations \cite{Sei08} and later used by  Auroux--Katzarkov--Orlov \cite{AKO06} in the context of mirror symmetry of Del Pezzo surfaces.

Yet smooth stable generalized complex structures have a clear shortcoming: no Poisson Fano manifold is smooth stable generalized complex in real dimension greater than four \cite{GP13}. Another quirk of the theory is that even though families (1)~to (4)~are all connected sums, until now that seemed to happen nearly by accident. Those manifolds were obtained performing surgeries on minimal surfaces followed by a computation to determine their diffeomorphism type. The obvious question here is whether there is a connected sum construction within the class of smooth stable generalized complex structures.

The present paper aims to tackle both of these issues. We do so by considering generalized complex structures whose anticanonical section is transverse to zero with self-crossings, hence are more singular than the original smooth stable case. Allowing for self-crossings is a natural weakening of the stability condition. In algebraic geometry one often makes no distinction between holomorphic sections transverse to zero and sections transverse to zero with self-crossings. This is also completely analogous to the move, in real Poisson geometry, from log-symplectic to log-symplectic with normal self-crossings \cite{GLPR17,MS18}.

We call these new structures just ``stable", without the adjective ``smooth'' to indicate that the zero locus of the anticanonical section is no longer a smooth embedded submanifold. An immediate consequence of the definition is that $\mathbb{C} P^{2n}$ is a stable generalized complex manifold and hence, above dimension four, this is a genuine weakening of the original stable condition.

This change in the singular behaviour of the anticanonical section is small enough that much of the theory developed in \cite{CG17} for smooth stable structures has a direct extension: we define complex logarithmic tangent bundle and extend the definition of elliptic tangent bundle. The existence of a stable generalized complex structure is equivalent to the existence of a symplectic structure on the elliptic tangent bundle satisfying certain cohomological conditions. This fact allows us to study stable generalized complex structures using symplectic techniques. 

Armed with symplectic techniques, we ask whether our results have any bearing on smooth stable \gcss. Dimension four turns out to be special:

\begin{namedtheorem}[Theorem \ref{th:smoothing}]
Any four-dimensional stable generalized complex structure can be deformed into a smooth stable \gcs.
\end{namedtheorem}

The study of symplectic structures on the elliptic tangent bundle which fail to satisfy only some of the conditions needed to produce a stable generalized complex structure turns out to be a fruitful detour. Indeed, we show that $\mathbb{C} P^2$, $\overline{\mathbb{C} P^2}$ and $S^4$ all admit symplectic structures which fail to be of generalized complex type at $0$ or $2$ ($\mathbb{C} P^2$), $1$ or $3$ ($\overline{\mathbb{C} P^2}$) and 1 ($S^4$) points. The existence of this structure in $S^4$ has a particularly remarkable consequence, namely, that we can develop a connected sum operation for these manifolds (c.f.~Theorem \ref{thm:ellipticglueing}) and by keeping track of the number of ``problem" points we can also determine when the resulting structure is stable generalized complex. The outcome is that we can extend the families (1)~to (4)~above and prove directly that:

\begin{namedtheorem}[Theorem \ref{th:examples}]
The manifolds in the following two families admit stable generalized complex structures:
\begin{enumerate}
\item  $\# n (S^2\times S^2)\# \ell (S^1\times S^3)$, with $n,\ell \in \nn$;
\item $\# n \mathbb{C} P^2 \#m \overline{\mathbb{C} P^2}\# \ell(S^1 \times S^3)$, with $n,m,\ell \in \nn$,
\end{enumerate}
as long as $1 - b_1 + b_2^+$ is even and the Euler characteristic is non-negative.
\end{namedtheorem}

Notice that if $1 - b_1 + b_2^+$ is odd for a four-manifold $M$, then $M$ does not admit any generalized complex structure as it is not even almost complex by \cite{HH58} or \cite[Theorem 1.4.13]{GS99}. The requirement that the Euler characteristic is positive, on the other hand, seems to be more of a limitation of our methods.

\vskip6pt
\noindent
{\bf Organisation of the paper.}
This paper is organised as follows. In Section~\ref{sec:divisors} we introduce self-crossing complex and elliptic divisors, the basic geometric objects that allow us to develop the theory of stable generalized complex structures. In Section~\ref{sec:algebroids} we will introduce the Lie algebroids induced by these divisors, which are the spaces where stable generalized complex structures become the more amenable symplectic structures. In Section \ref{sec:generalizedcomplex} we introduce self-crossing stable generalized complex structures and show that they are equivalent to a certain class of elliptic symplectic structures. In Section~\ref{sec:4d} we focus on four-dimensional structures. Here we prove a normal form theorem for self-intersection points in the divisor and show that a stable structure can be deformed into a smooth one (Theorem~\ref{th:smoothing}). In Section~\ref{sec:connectedsum} we show that one can perform connected sums of stable generalized complex structures (Theorem~\ref{thm:ellipticglueing}) and in Section~\ref{sec:examples} we provide concrete examples obtained via connected sum and prove Theorem~\ref{th:examples}.

\vskip6pt
\noindent
{\bf Acknowledgements.}
We thank Eduard Looijenga for useful conversations regarding complex log divisors and Ornea and Vuletescu for pointing us towards \cite{MR1760667} for the part of the argument in Remark \ref{rem:not complex}. RK was supported by ERC consolidator grant 646649 ``SymplecticEinstein''. AW was supported by the NWO through the Utrecht Geometry Centre Graduate Programme.

\section{Self-crossing divisors}\label{sec:divisors}
This section covers the basic definitions and properties of the singularities we will encounter. We start by recalling the definition of a divisor, before introducing the complex log divisors we are mostly interested in. After discussing their zero sets in some detail, we turn to elliptic divisors. These can be induced from complex log divisors, and will play a large role throughout the paper. Finally we describe the relation between complex log and elliptic divisors in full detail. 
\subsection{Divisors}
In this section we define the objects which will govern the singularities of the geometric structures that are to come. The use of divisors in algebraic geometry is extremely common, but our terminology differs slightly:
\begin{definition}
A \textbf{complex divisor} on a manifold $M$ is a pair $(L,\sigma)$ where $L \rightarrow M$ is a complex line bundle and $\sigma \in \Gamma(L)$ is a section with nowhere dense zero set.
\end{definition}
Given a divisor, we can view the section as a map $\sigma\colon \Gamma(L^*) \rightarrow C^{\infty}(M;\cc)$ and obtain a complex ideal $I_{\sigma} := \sigma(\Gamma(L^*))$. Divisors admit products, which are obtained by performing tensor products. For clarity, we have that $(L,\sigma) \otimes (L',\sigma') = (L\otimes L', \sigma\otimes\sigma')$ and hence $I_{\sigma\otimes\sigma'} = I_\sigma \cdot I_{\sigma'}$.
\begin{definition}
Let $(M,(L,\sigma))$ and $(N,(L',\sigma'))$ be manifolds with divisors. A smooth map $\varphi\colon M \rightarrow N$ is a \textbf{morphism of divisors} if $\varphi^*I_{\sigma'} = I_{\sigma}$, where the left-hand side denotes the ideal generated by all pullbacks, and is an \textbf{isomorphism} if $M = N$ and $\varphi = \text{id}_M$.
\end{definition}
The ideal defined by the section actually completely captures the divisor up to isomorphism, so that we will often use divisors and their corresponding ideals interchangeably. Moreover, we will sometimes denote a product of ideals by $I \otimes I'$ instead of $I \cdot I'$. We will mainly be interested in a specific class of divisors:
\begin{definition}
A complex divisor $(L,\sigma)$ on a manifold $M$ is a \textbf{smooth complex log divisor} if $\sigma$ vanishes transversely. Its \textbf{vanishing locus} is denoted by $D := \sigma^{-1}(\set{0})$.
\end{definition}
By transversality, the vanishing locus of a smooth complex log divisor is an embedded submanifold of codimension two. To proceed further we will need the following notion.
\begin{definition}\label{def:1ideal}
A collection of $I_1,\ldots,I_j \subset \Omega^0(M;\mathbb{C})$ of locally principal ideals is \textbf{functionally independent} if there exist local generators $f_i$ of $I_i$ which are functionally independent at their common zero set. That is, for all multi-indices  $(i_1,\ldots,i_k)$ with length smaller or equal to $j$ we have
\begin{equation*}
d_pf_{i_1} \wedge d _p\bar{f_{i_1}}\wedge \cdots \wedge d_pf_{i_k} \wedge  d _p\bar{f_{i_k}} \neq 0, \qquad \text{for all } p \in \cap_{l=1}^k f_{i_l}^{-1}(\set{0}).\qedhere
\end{equation*}
\end{definition}
\begin{definition}\label{def:1}
A \textbf{self-crossing complex log divisor} is a complex divisor $(L,\sigma)$, such that for every point $p \in M$, there exists a neighbourhood $U$ of $p$ such that
\begin{equation*}
I_{\sigma}(U) = I_1\cdot \ldots \cdot I_j,
\end{equation*}
where $I_1,\ldots,I_j$ are functionally independent smooth complex log divisors on $U$.
\end{definition}
\begin{remark}[Terminology]\label{rem:term}
Our definition of a smooth complex log divisor appears in \cite{CG17} without the prefix smooth attached. For brevity, we will often write ``complex log divisor'', which has to be understood to possibly have self-crossings. Whenever we deal with a smooth complex log divisor we will explicitly stress this.
\end{remark}
\begin{definition}
Let $I_D$ be a complex log divisor. For a given point $p \in M$, and a neighbourhood $U$ of $p$ let $n_U$ be the number $j$ as in Definition \ref{def:1}. The \textbf{intersection number of $p$} is the minimum of $n_U$ taken over all neighbourhoods of $p$. The \textbf{intersection number of the divisor} is the maximum of the intersection numbers of its points.
\end{definition}
\begin{example}\label{ex:mainexample} Let $\{I_{D_i}\}$ for $i = 1,\ldots, n$ be a collection of functionally independent smooth complex log divisors. Then their product $I_D := \otimes_{i=1}^nI_{D_i}$ defines a complex log divisor with intersection number $n$. We call such a divisor a \textbf{global normal crossing divisor.}
\end{example}
By definition every complex log divisor is locally of this form, which will often be used.
\begin{definition}
Given a complex log divisor $I_D$, we call a choice of local smooth complex divisors near a point as in Example \ref{ex:mainexample} a \textbf{local normal crossing}. 
\end{definition}
Note that up to the germ of a local isomorphism of divisors, the only choice in a local normal crossing for a given complex log divisor is the ordering of the smooth divisors.
\begin{example}\label{ex:complexlogcpn}
Let $\mathcal{O}(k)$ be the holomorphic line bundle on $\mathbb{C}P^n$ obtained as the $k$-fold tensor product of the dual of the tautological line bundle. Recall that sections of $\oo(k)$ can be identified with homogeneous polynomials of degree $k$ in $n+1$ variables. Under this identification we can view the polynomial $p := z_0\cdot \ldots \cdot z_n$ as a section of $\oo(n+1)$. We conclude that $(\mathcal{O}(n+1), p)$ defines a complex log divisor with intersection number $n$ on $\cc P^n$.
\end{example}
\begin{example}\label{ex:linebundledivisor}
Let $E \rightarrow M$ be a complex line bundle. Then $\Gamma((E^{1,0})^*) \subset C^{\infty}(E)$ generates an ideal $I$ on $E$. Locally, if $U \subset M$ is an open neighbourhood on which $\rest{E}{U}$ is trivialised, there exists a corresponding fibre coordinate $z$ on $U$ which generates $I$. We conclude that every complex line bundle carries a canonical smooth complex log divisor whose vanishing locus is the zero section $M \subset E$.
\end{example}
\begin{example}\label{rem:locform}
Let $(z_1,\ldots, z_j,x_{2j+1},\ldots,x_{2j+m})$ be coordinates on $\cc^j \times \rr^m$ and define smooth complex log divisors $I_{D_i} := \langle z_i \rangle$. Then the ideal $I_D = \otimes_{i=1}^jI_{D_i}$ is called the \textbf{standard complex log divisor with intersection number $j$ on $\cc^j \times \rr^m$}.
\end{example}
\begin{lemma}\label{lem:loccomplexlogiso}
If $I_D$ is a complex log divisor and $p \in M$ is a point of intersection number $j$, then $I_D$ is locally isomorphic to the standard complex log divisor around $p$.
\end{lemma}
\begin{proof}
Because the representatives $f_1,\ldots,f_j$ of the ideals of a local normal crossing vanish transversely and are functionally independent near $p$, they can be completed to a local coordinate system on an open neighbourhood of $p$. These coordinates provide the divisor isomorphism.
\end{proof}
The vanishing locus of a complex log divisor is not an embedded submanifold when its intersection number is larger then one, it is however immersed.
\begin{lemma}\label{lem:immersed}
Let $I_D$ be a complex log divisor on $M^{2j+m}$, and let $D$ be its vanishing locus. Then the vanishing locus $D$ is an immersed submanifold.
\end{lemma}
\begin{proof}
As being immersed is a local property, we need to show that every point in $D$ has a neighbourhood on which $D$ is immersed. By Lemma \ref{lem:loccomplexlogiso} it suffices to show that the vanishing locus of the standard complex log divisor is an immersed submanifold. The obvious map of inclusions of the different coordinate planes $\bigcup_{i=1}^j \cc^{j-1}\times \rr^m \rightarrow \cc^j \times \rr^m$ provide this immersion.
\end{proof}
Intuitively, the vanishing locus of a complex log divisor is the immersion of a manifold, $\tilde D$, obtained from $D$ by duplicating the intersection locus and separating the strands whenever self-crossings occur. Here we need to introduce some subtle language variation to distinguish between different meanings of the word {\it component}: a {\bf connected component} of $D$ is just that, a connected component of $D$ as a subspace of $M$, while a {\bf component} of $D$ is the image of a connected component of $\tilde{D}$.

The degeneracy locus is not only an immersed submanifold but it is also stratified by embedded smooth submanifolds.
\begin{definition}\label{rem:filtration}
Let $I_D$ be a complex log divisor with intersection number $n$ on a manifold $M$. Given $1 \leq j \leq n$, the set of points with intersection number \emph{at least $j$} will be denoted by $D(j)$. These sets induce a filtration on $M$, namely
\begin{equation*}
M = D(0) \supset D = D(1) \supset D(2) \supset \cdots \supset D(n).
\end{equation*}
We will call this filtration the \textbf{intersection stratification} of $M$ induced by $I_D$. The strata of this stratification are denoted by
\begin{equation*}
C(i):= D(i)\backslash D(i+1),
\end{equation*}
and consist of the points with intersection number \emph{exactly} $i$, and each have codimension $2i$.
\end{definition}
The following is immediate, and will be used without further mention throughout this paper.

\begin{lemma}\label{lem:restricdiv}
Let $I_D$ be a complex log divisor on a manifold $M$ with intersection number at least $i$. Then $\rest{I_D}{M\backslash D(i+1)}$ is a complex log divisor on $M \backslash D(i+1)$ with intersection number $i$.
\end{lemma}
With this in mind we can verify that the definition of the stratification makes sense.

\begin{lemma}\label{lem:honeststrata}
The filtration from Definition \ref{rem:filtration} defines a smooth stratification on $M$.
\end{lemma}
\begin{proof}
We first note that the highest codimension stratum $D(n)$ is a smooth submanifold by the regular value theorem. Next, the subset $C(i) = D(i)\backslash D(i+1)$ is the highest codimension stratum of the restricted divisor to $M\backslash D(i+1)$ and is therefore smooth. Finally the filtration induces a stratification precisely because we have the local form as described in Lemma \ref{lem:loccomplexlogiso}.
\end{proof}

\subsection{Elliptic divisors}
We now introduce self-crossing elliptic divisors. These divisors arise as the real part of complex log divsors, but can be defined independently. 
\begin{definition}[\cite{CG17}]
A \textbf{smooth elliptic divisor} $(R,q)$ consists of a real line bundle $R$ together with a section $q \in \Gamma(R)$, whose zero set $D = q^{-1}(0)$ is a codimension-two submanifold along which its normal Hessian is positive definite.
\end{definition}
The normal Hessian of the section $q$ is the section ${\rm Hess}(q) \in \Gamma(D;{\rm Sym}^2 N^*D \otimes R)$ containing the leading term of its Taylor expansion. As was done for complex log divisors we usually refer to the divisor by its corresponding ideal. Note that a given codimension-two submanifold may carry multiple smooth elliptic divisor structures. To proceed we again consider divisors which are local normal crossings of smooth elliptic divisors.
\begin{definition}\label{def:elli}
A \textbf{self-crossing elliptic divisor} on a manifold $M$ is a real divisor $I_{\abs{D}}$ such that for every $p \in M$ there exists an open neighbourhood $U$ of $p$ such that 
\begin{equation*}
I_{\abs{D}}(U) = I_{\abs{D_1}}\cdot \ldots \cdot I_{\abs{D_j}}.
\end{equation*}
Here the $I_{\abs{D_i}}$ are smooth elliptic divisors on $U$ whose zero loci $D_i$ intersect transversely\footnote{A collection of submanifolds $\set{D_i}$ of is said to intersect transversely at a point $x \in M$ if 
\begin{equation*}
\text{codim}(\bigcap_i T_xD_i) = \sum_i \text{codim}(T_xD_i).
\end{equation*}
}.
\end{definition}
As discussed in Remark \ref{rem:term} we will often omit the prefix ``self-crossing'' and instead add ``smooth'' when referring to an elliptic divisor in the sense of \cite{CG17}. 
\begin{remark}\label{rem:hessindep}
The condition that the loci of the smooth elliptic divisors are transverse is equivalent to the following statement: for all local generators $f_i$ of $I_{\abs{D_i}}$ we have that
\begin{equation*}
\ker \text{Hess}_p (f_{i_1}) \cap \cdots \cap \ker \text{Hess}_p (f_{i_k}) \qquad \text{for all } p \in \cap_{l=1}^k f_{i_l}^{-1}(\set{0})
\end{equation*}
has minimal dimension for all multi-indices $(i_1,\ldots,i_k)$ of length smaller or equal than $k$.
\end{remark}
Many of the notions we defined for complex log divisors with self-crossings can also be defined for elliptic divisors with self-crossings. In particular they have an intersection number and an induced stratification, which will be denoted in the same manner as in the complex log case.

An important class of elliptic divisors arises from complex log divisors:
\begin{example}\label{ex:complextoelli}
If $I_D$ is a complex log divisor, then $I_D \cdot \bar{I}_D$ is invariant under conjugation. Therefore, there exists a real ideal $I_{\abs{D}}$ such that $I_{\abs{D}} \otimes \cc = I_D \cdot \bar{I}_D$. By definition, locally $I_{D} = I_{D_1} \cdot \ldots \cdot I_{D_n}$, where the $I_{D_i}$ are smooth complex log divisors with transverse zero loci. Therefore we see that
\begin{equation*}
I_{\abs{D}}(U) \otimes \cc = (I_{D_1} \cdot \bar{I}_{D_1}) \cdot \ldots \cdot (I_{D_n}\cdot \bar{I}_{D_n}),
\end{equation*}
hence $I_{\abs{D}}(U)$ is given as the product of smooth elliptic divisors with transverse vanishing loci. We conclude that $I_{\abs{D}}$ is an elliptic divisor, and call it \textbf{the elliptic divisor induced by a complex log divisor}.
\end{example}

The following is completely analogous to the complex log setting (Example~\ref{ex:mainexample}):
\begin{example}\label{ex:mainexampleelliptic}
Given smooth complex elliptic divisors $I_{\abs{D_i}}$ for which the vanishing loci $D_i$ are transverse, we have that $I_{\abs{D}} := \otimes_{i=1}^n I_{\abs{D_i}}$ defines an elliptic divisor with intersection number $n$. We call this a \textbf{global normal crossing} elliptic divisor.
\end{example}
By definition, every elliptic divisor is locally of the above form, warranting the following.
\begin{definition}
Given an elliptic divisor $I_{\abs{D}}$, we call a choice of local smooth complex divisors near a point as in Definition \ref{def:elli} a \textbf{local normal crossing}.
\end{definition}
\begin{example}\label{rem:locformelliptic}
Let $(x_1,y_1,\ldots, x_j,y_j,x_{j+1},\ldots,x_m)$ be coordinates on $\rr^{2j} \times \rr^m$ and define smooth elliptic divisors $I_{\abs{D_i}} = \langle x_i^2+y_i^2 \rangle$. We call $I_{\abs{D}}:= \otimes_{i=1}^n I_{\abs{D_i}}$ the \textbf{standard elliptic divisor of intersection number $j$ on $\rr^{2j} \times \rr^m$}.
\end{example}
Using the Morse--Bott lemma we can locally put an elliptic divisor in standard form.
\begin{lemma}\label{lem:locMB}
Let $I_{\abs{D}}$ be an elliptic divisor on a manifold $M$ and let $p \in M$ have intersection number $j$. Then $I_{\abs{D}}$ is locally isomorphic to the standard elliptic divisor of intersection number $j$ on $\rr^{2j} \times \rr^m$, where $\dim M = 2j + m$.
\end{lemma}
\begin{proof}
To simplify notation we will consider $j = 2$ as the proof in the general case is identical. Let $U$ be an open neighbourhood of a point $p$ of intersection number $2$ and let $I_{\abs{D_1}},I_{\abs{D_2}}$ be a local normal crossing. Let $f_1,f_2$ be representatives of $I_{\abs{D_1}},I_{\abs{D_2}}$ respectively, and apply the Morse--Bott lemma to obtain coordinates $(x_1,y_1,x_2,y_2,z_3,\ldots)$, $(\tilde{x}_1,\tilde{y}_1,\tilde{x}_2,\tilde{y}_2,\tilde{z}_3,\ldots)$ on neighbourhoods $U_1,U_2$ of $p$ respectively such that $f_1 = x_1^2 + y_1^2$ and $f_2 = \tilde{x}_2^2 + \tilde{y}_2^2$. Consider
\begin{equation*}
\Phi = (x_1,y_1,\tilde{x}_2,\tilde{y}_2,z_3,\ldots)\colon U_1 \cap U_2 \rightarrow \rr^n,
\end{equation*}
and 
\begin{equation*}
F\colon \rr^n \rightarrow \rr \quad (x_1,\ldots,x_n) \mapsto x_1^2+x_2^2+x_3^2+x_4^2.
\end{equation*}
Then $f_1+f_2 = F \comp \Phi$. Note that the normal Hessian of $F$ is non-degenerate and that the normal Hessian of $f_1 + f_2$ is non-degenerate by the independence condition in Remark \ref{rem:hessindep}. Because
\begin{equation*}
\Hess_{f_1+f_2} = \Hess_{F\comp \Phi} = \Hess_F(\Phi_*\cdot,\Phi_*\cdot),
\end{equation*}
we conclude that $\Phi_*$ must be injective on $N(D_1 \cap D_2)$. Therefore $\Phi$ must be a local diffeomorphism and, using the inverse function theorem and possibly shrinking the domain of definition, we conclude that $(x_1,y_1,\tilde{x}_2,\tilde{y}_2,z_3,\ldots)$ gives the required coordinate system.
\end{proof}

It follows from this lemma that, just as for complex divisors, the vanishing locus of an elliptic divisor, $(R,q)$, is an immersed submanifold, $D$, with transverse self crossings. Further since $q$ is a trivialization of $R$ over $M\backslash D$ and $D$ has codimension two, $R$ is also trivializable and $q$ determines a preferred orientation for $R$. Therefore,  one can define an elliptic divisor alternatively as the ideal generated by a function $f\colon M \to \rr_+$ whose zeros are locally of the form
\[ f(x_1,y_1,\dots,x_k,y_k,x_{k+1},\dots, x_{n}) = (x_1^2 + y_1^2)\dots (x_k^2 + y_k^2).\]
\subsubsection{Examples}
A class of examples of elliptic divisors arises from toric geometry.
\begin{example}\label{ex:toric}
Let $\mu\colon M \rightarrow \rr^n$ be a toric manifold with moment polytope $\Delta$ and let $\lambda_i \in \rr^n$ be vectors transverse to its faces. If we denote $f_{\lambda_i}\colon \rr^n \rightarrow \rr,  x \mapsto \inp{x,\lambda_i}$, then the ideal
\begin{equation}\label{eq:toric}
I := \inp{(\mu\comp f_{\lambda_1})\cdot \ldots \cdot (\mu \comp f_{\lambda_{n+1}})}
\end{equation}
defines an elliptic divisor on $M$ with intersection number $n$ and vanishing locus $\mu^{-1}(\partial \Delta)$. Namely, the functions $f_{\lambda_i}$ are linear and have the faces of the moment polytope as zero sets. Because $M$ is toric, the components of the moment map are definite Morse--Bott functions, hence so are the compositions $\mu \comp f_{\lambda_i}$. We conclude that the ideal $I$ defines an elliptic divisor.
\end{example}
An explicit case of the setting of the above example occurs on the manifold $\cc P^n$.
\begin{example}\label{ex:divisorcpn}
Consider $\cc P^n$ with the moment map
\begin{equation}\label{eq:momentcpn}
\mu\colon \mathbb{C}P^n \rightarrow \Delta \qquad [z_0:z_1:\cdots:z_n] \mapsto \frac{(\abs{z_0}^2,\abs{z_1}^2,\ldots,\abs{z_{n-1}}^2)}{\abs{z_0}^2+ \abs{z_1}^2 +\cdots + \abs{z_n}^2}.
\end{equation}
Here $\Delta = \set{(x_1,\ldots,x_n) \in \rr^n : x_i \leq 1}$ denotes the moment polytope. Proceeding as in Example \ref{ex:toric} endows $\cc P^n$ with the structure of an elliptic divisor. Note that this is the elliptic divisor induced by the complex log divisor of Example \ref{ex:complexlogcpn}.
\end{example}
\subsection{Elliptic versus complex log divisors}
As we have seen in Example \ref{ex:complextoelli}, a complex log divisor $(L,\sigma)$ induces an elliptic divisor. The complex log divisor also induces a complex structure on the normal bundle of $C(1)$ via the isomorphism\footnote{Because $\sigma$ vanishes transversely outside of $D(2)$, the normal derivative is an isomorphism.}
\begin{equation*}
	\rest{d^{\nu}\sigma}{C(1)}\colon NC(1) \rightarrow \rest{L}{C(1)}.
\end{equation*}
Note however that this complex structure depends on the particular choice of section $\sigma$. The orientation on $NC(1)$ induced by these possibly different complex structures is independent of such choices. These two pieces of information, the elliptic divisor and the co-orientation, completely determine the complex log divisor up to isomorphism. This statement was already mentioned in the smooth case in \cite[Section 1.2]{CG17}, but appeared there without proof. 
\begin{proposition}\label{prop:ellivscomplex}
Let $M$ be a manifold. The association
\begin{equation*}
(L,\sigma) \mapsto ((R,q),\mathfrak{o})
\end{equation*}
which sends a complex log divisor on $M$ to its associated elliptic divisor, together with the induced co-orientation of $C(1)$, induces a bijection of isomorphism classes of complex log divisors and isomorphism classes of elliptic divisors with chosen co-orientation of $C(1)$.
\end{proposition}
\begin{proof}
Let $I_1,I_2$ be two complex log divisors. Assume that $I_1$ and $I_2$ both induce the same elliptic ideal and the same co-orientation on the normal bundle to $C(1)$. We have to prove that $I_1 = I_2$. We will proceed via several steps.

\textbf{Injectivity for smooth divisors}: We first prove injectivity for smooth divisors. Let $z,w$ be complex coordinates such that $I_1 = \inp{z}$ and $I_2 = \inp{w}$. By assumption we have $I_1 \otimes \bar{I_1} = I_2 \otimes \bar{I_2}$ and thus there exists a nowhere vanishing function $g \in C^{\infty}(M,\rr)$ such that $g z\bar{z} =  w\bar{w}$. The fact that $I_1$ and $I_2$ induce the same co-orientation ensures the existence of some strictly positive function $h \in C^{\infty}(D;\rr)$ such that $h d^{\nu}z \wedge d^{\nu} \bar{z} = d^{\nu}w \wedge d^{\nu} \bar{w}$.
\begin{claim*} We have that $h= \rest{g}{D}$.
\end{claim*}
\begin{proof}[Proof of claim]
By taking the derivative of $gz\bar{z} = w \bar{w}$ with respect to $\bar{w}$, we obtain that
\begin{equation*}
\pd{g}{\bar{w}}z\bar{z}+g\pd{z}{\bar{w}}\bar{z}+ gz\pd{\bar{z}}{\bar{w}} = w.
\end{equation*}
By taking the derivative of this equation with respect to $z$, we get
\begin{equation*}
\pd{^2g}{z\partial \bar{w}}z\bar{z}+ \pd{g}{\bar{w}}\bar{z}+\pd{g}{z}\pd{z}{\bar{w}}\bar{z}+\pd{g}{z}z\pd{\bar{z}}{\bar{w}} + g\pd{\bar{z}}{\bar{w}} = \pd{w}{z}.
\end{equation*}
In particular we find $\rest{g}{D}\rest{\pd{\bar{z}}{\bar{w}}}{D} = \rest{\pd{w}{z}}{D}$. Noting that
\begin{equation*}
\inp{d^{\nu}w\wedge d^{\nu} \bar{w}, \partial_z \wedge \partial_{\bar{w}}} = \pd{w}{z} = h\pd{\bar{z}}{\bar{w}},
\end{equation*}
we can combine these facts to conclude that $\rest{g}{D} = h$.
\end{proof}
Continuing our main line of reasoning, in order to show that $w \in I_1$ we are going to invoke Malgrange's Theorem, \cite[Theorem 1.1]{M66}, which states that $w \in I_1$ if and only if its formal power series with respect to $z$ and $\bar{z}$ is divisible by $z$. We expand $w$ as a power series in $z$ and $\bar{z}$: 
\begin{equation*}
w = a_{10}z + a_{01}\bar{z} + \sum_{i + j \geq 2} a_{ij} z^i \bar{z}^j.
\end{equation*}
\begin{claim*}
We have that $\abs{a_{10}}^2 = \abs{a_{01}}^2 + h$.
\end{claim*}
\begin{proof}[Proof of claim]
Using $h d^{\nu}z \wedge d^{\nu} \bar{z} = d^{\nu}w \wedge d^{\nu} \bar{w}$, we find 
\begin{align*}
\inp{d^{\nu}w\wedge d^{\nu} \bar{w}, \partial_z \wedge \partial_{\bar{z}}} &=  h,
\end{align*}
but also
\begin{align*}
\inp{d^{\nu}w\wedge d^{\nu} \bar{w}, \partial_z \wedge \partial_{\bar{z}}} &=\pd{w}{z}\pd{\bar{w}}{\bar{z}}-\pd{\bar{w}}{z}\pd{w}{\bar{z}}\\
&= \abs{\pd{w}{z}}^2 - \abs{\pd{\bar{w}}{z}}^2\\
&= \abs{a_{10}}^2 + \abs{a_{01}}^2.\qedhere
\end{align*}
\end{proof}
Knowing this, we can express the product $w \bar{w}$ as
\begin{align*}
w\bar{w} &= (\abs{a_{10}}^2+\abs{a_{01}}^2)z \bar{z} + a_{10}\bar{a}_{01} z^2 + a_{10}\sum_{i+ j \geq 2} \bar{a_{ij}} \bar{z}^{i}z^{j+1}\\ 
&+ a_{01}\bar{a}_{10}\bar{z}^2 +  a_{01}\sum_{i+ j \geq 2} \bar{a}_{ij}\bar{z}^{i+1} z^j + \bar{a}_{10}\sum_{i+ j \geq 2} a_{ij}z^{i} \bar{z}^{j+1}\\
 &+ \bar{a}_{01}\sum_{i+ j \geq 2} a_{ij}z^{i+1} \bar{z}^{j} + \sum_{\substack{i+j \geq 2 \\ k+l \geq 2}} a_{ij} \bar{a}_{kl} z^{i+k}\bar{z}^{j+l},
\end{align*}
and because $\abs{a_{10}}^2 = \abs{a_{01}}^2 + h$, and $w\bar{w} =g  z\bar{z}$, we see that 
 \begin{align*}
0 &= (h-g+2\abs{a_{01}}^2)z \bar{z} + a_{10}\bar{a}_{01} z^2 + a_{10}\sum_{i+ j \geq 2} \bar{a_{ij}} \bar{z}^{i}z^{j+1}\\ 
&+ a_{01}\bar{a}_{10}\bar{z}^2 +  a_{01}\sum_{i+ j \geq 2} \bar{a}_{ij}\bar{z}^{i+1} z^j + \bar{a}_{10}\sum_{i+ j \geq 2} a_{ij}z^{i} \bar{z}^{j+1}\\
 &+ \bar{a}_{01}\sum_{i+ j \geq 2} a_{ij}z^{i+1} \bar{z}^{j} + \sum_{\substack{i+j \geq 2 \\ k+l \geq 2}} a_{ij} \bar{a}_{kl} z^{i+k}\bar{z}^{j+l}.
\end{align*}
By expanding $h-g$ as a power series, and because $\rest{g}{D} = h$, we see that $h-g = \sum_{i+j \geq 1}b_{ij}z^i\bar{z}^j$. In conclusion we have obtained the following equality of power series:
 \begin{align*}
0 &= \left(\sum_{i+j \geq 1}b_{ij}z^i\bar{z}^j + 2\abs{a_{01}}^2\right)z \bar{z} + a_{10}\bar{a}_{01} z^2 + a_{10}\sum_{i+ j \geq 2} \bar{a_{ij}} \bar{z}^{i}z^{j+1}\\ 
&\qquad+ a_{01}\bar{a}_{10}\bar{z}^2 +  a_{01}\sum_{i+ j \geq 2} \bar{a}_{ij}\bar{z}^{i+1} z^j + \bar{a}_{10}\sum_{i+ j \geq 2} a_{ij}z^{i} \bar{z}^{j+1}\\
&\qquad+ \bar{a}_{01}\sum_{i+ j \geq 2} a_{ij}z^{i+1} \bar{z}^{j} + \sum_{\substack{i+j \geq 2 \\ k+l \geq 2}} a_{ij} \bar{a}_{kl} z^{i+k}\bar{z}^{j+l}.
\end{align*}
Therefore all the coefficients of this power series need to vanish. The term of degree 1 in $z$ and degree 1 in $\bar{z}$ is given by $2\abs{a_{01}}z\bar{z}$, hence we conclude that $a_{01} = 0$. The degree-$n$ term in $\bar{z}$ is given by $\bar{a}_{10}a_{0,n-1} \bar{z}^n$. Hence we can conclude that $\bar{a}_{0,i} = 0$ for all $i \geq 0$. Therefore the formal power series of $w$ is divisible by $z$, and we conclude that $w \in \inp{z}$. By symmetry we conclude that $\inp{w} = \inp {z}$, from which we conclude that $I_1 = I_2$ which finishes this part of the proof.\\
\textbf{Injectivity for global normal crossings}: Suppose that $I_{\abs{D}}= \otimes_i I_{\abs{D_i}}$ is a global normal crossing elliptic divisor. Let $I_{D^1}$ and $I_{D^2}$ be two complex log divisors which induce $I_{\abs{D}}$, which both need to be global normal crossing divisors because $I_{\abs{D}}$ is. Moreover, assume that they induce the same co-orientation on $C(1)$. This implies that they induce the same co-orientation on each of the zero loci $D_i$. After reordering, we may assume that
\begin{equation*}
 	I_{\abs{D^1_i}} = I_{\abs{D_i}} = I_{\abs{D^2_i}}.
\end{equation*}
Thus for each $i$, we have that $I_{\abs{D^1_i}}$ and $I_{\abs{D^2_i}}$ induce the same smooth elliptic divisor and the same co-orientation and therefore $I_{\abs{D^1_i}} = I_{\abs{D^2_i}}$ by the above, from which the result follows.\\
\textbf{Injectivity for general divisors}: Because $I_{\abs{D}}$ is locally a normal crossing of elliptic divisors, around every point $x \in D$ we can find an open neighbourhood $U$ such that $\rest{I_{\abs{D}}}{U}$ is a normal crossing of elliptic divisors. Therefore, if two general complex log divisors induce the same elliptic ideal and same co-orientation, we can use the injectivity for global normal crossing divisors and argue locally to prove that the complex log divisors must be isomorphic.\\
\textbf{Surjectivity for smooth divisors}: Let $I_{\abs{D}}$ be a smooth elliptic divisor and $\mathfrak{o}$ an orientation of $ND$. Given a choice of representative $f \in I_{\abs{D}}$, we can view $\Hess^{\nu}f \in \Gamma(S^2N^*D)$ as a metric on $N^*D$. We use this metric to transport the orientation of $ND$ to an orientation on $N^*D$. The orientation together with the metric induces a complex structure on $ND$, and hence a complex log divisor structure on $ND$ by Example \ref{ex:linebundledivisor}. This complex log divisor induces $I_{\abs{D}}$, which proves surjectivity for smooth divisors.\\
\textbf{Surjectivity for global normal crossings}: Let $I_{\abs{D}} = \otimes I_{\abs{D_i}}$ be a global normal crossing elliptic divisor. Because $\rest{NC(1)}{D_i \backslash D(2)} = \rest{ND_i}{D_i \backslash D(2)}$ and $D(2)$ is codimension two in $D_i$ we conclude that each $ND_i$ is orientable. Therefore, by the above, there exist complex log divisors $I_{D_i}$ such that $I_{\abs{D_i}} = I_{D_i} \otimes I_{\bar{D}_i}$. We conclude that 
\begin{equation*}
I_{\abs{D}}\otimes \cc = \otimes_i (I_{D_i} \otimes I_{\bar{D_i}}) = (\otimes_i I_{D_i}) \otimes \bar{(\otimes_i I_{D_i})}.
\end{equation*}
\textbf{Surjectivity for general divisors}: Because $I_{\abs{D}}$ is locally a global normal crossing of elliptic divisors, for every point $x \in D$ we can find an open neighbourhood $U$ such that $\rest{I_{\abs{D}}}{U}$ is a global normal crossing. Therefore, by the previous part there exists a complex log divisor $I_{D_U}$ which induces $\rest{I_{\abs{D}}}{U}$. Let $\mathcal{U}$ be an open cover of $M$, such that $I_{\abs{D}}$ is a global normal crossing on each open in the cover and construct complex log divisors inducing $I_{\abs{D}}$ on each of these opens. Let $U,U' \in \mathcal{U}$ and let $I_U$ and $I_{U'}$ be complex log divisors inducing $\rest{I_{\abs{D}}}{U}$ and $\rest{I_{\abs{D}}}{U}$. On the overlap $U\cap U'$ both $\rest{I_{U'}}{U\cap U'}$ and $\rest{I_U}{U\cap U'}$ induce the same elliptic ideal, by the above we therefore have $\rest{I_U}{U\cap U'} = \rest{I_{U'}}{U\cap U'}$. We conclude that the local complex log divisors glue to a global complex log divisor which induces $\rest{I_{\abs{D}}}{U}$, which finishes the proof.
\end{proof}
Motivated by this result we define the following.
\begin{definition}
An elliptic divisor $I_{\abs{D}}$ is \textbf{co-orientable} if $C(1)$ is co-orientable.
\end{definition}

\section{Lie algebroids associated to self-crossing divisors}\label{sec:algebroids}
In this section we introduce the Lie algebroids associated to self-crossing complex log and elliptic divisors. Much of this section follows along the same lines as \cite{CG17}. The Lie algebroids will be defined by imposing that their sections interact appropriately with the divisors, in that they must preserve the divisor ideals.
\subsection{Complex log tangent bundle}
\begin{lemma}\label{lem:logtan}
Let $I_D$ be a complex log divisor on a manifold $M$. The complex vector fields preserving the complex ideal $I_D$ are sections of a complex Lie algebroid $\mathcal{A}_D$.
\end{lemma}
\begin{proof}
By Lemma \ref{lem:loccomplexlogiso}, at every point there exist coordinates $(z_1,\ldots,z_n,x_i)$ such that locally $I_D = \inp{z_1\cdot\ldots\cdot z_n}$. One can readily check that in these coordinates the vector fields preserving $I_D$ are generated by
\begin{equation*}
\set{ z_1 \partial_{z_1},\partial_{\bar{z}_1}\ldots, z_n \partial_{z_n},\partial_{\bar{z}_n},\partial_{x_i}}
\end{equation*}
and therefore form a locally free sheaf which, by the Serre--Swan theorem, correspond to sections of a vector bundle, $\mathcal{A}_D$. The Lie bracket of vector fields induces a bracket on the sections of $\mathcal{A}_D$ making it into a Lie algebroid.
\end{proof}
\begin{definition} The Lie algebroid $\mathcal{A}_D$ is the \textbf{complex log tangent bundle} given by $I_D$.
\end{definition}
There is a local description of the complex log tangent bundle, whose proof is immediate.
\begin{lemma}\label{lem:complexlogfiberproduct}
Let $I_D$ be a complex log divisor on a manifold $M$, and let $I_{D_1},\ldots,I_{D_n}$ be a choice of local normal crossing on some open $U$. Then $\mathcal{A}_D$ is given by a repeated fiber product:
\begin{equation*}
\rest{\mathcal{A}_D}{U} \simeq \mathcal{A}_{D_1} \times_{T_{\cc}M} \times \cdots \times_{T_{\cc}M} \mathcal{A}_{D_n}.
\end{equation*}
\end{lemma}
\subsection{Elliptic tangent bundle}
There is also a Lie algebroid associated to an elliptic divisor.
\begin{lemma}\label{lem:ellialg}
Let $I_{\abs{D}}$ be an elliptic divisor on a manifold $M$. The vector fields preserving the ideal $I_{\abs{D}}$ are sections of a Lie algebroid $\mathcal{A}_{\abs{D}}$.
\end{lemma}
\begin{proof}
By Lemma \ref{lem:locMB} there exists coordinates $(r_1,\theta_1,\ldots,r_n,\theta_n,x_i)$ such that locally we have $I_{\abs{D}} = \inp{r_1^2\cdot\ldots\cdot r_n^2}$ and in which the vector fields preserving $I_{\abs{D}}$ are generated by
\begin{equation*}
\set{r_1\partial_{r_1},\partial_{\theta_1},\ldots,r_n\partial_{r_n},\partial_{\theta_n},\partial_{x_i}}.
\end{equation*}
This collection forms a locally free sheaf, hence there exists a Lie algebroid $\mathcal{A}_{\abs{D}}$ whose sections are the vector fields preserving $I_{\abs{D}}$.
\end{proof}
\begin{definition}\label{def:ellialg}
The Lie algebroid $\mathcal{A}_{\abs{D}}$ is the \textbf{elliptic tangent bundle} associated to $I_{\abs{D}}$.
\end{definition}

Similar to Lemma \ref{lem:complexlogfiberproduct} we have the following local description of the elliptic tangent bundle.
\begin{lemma}\label{lem:locfibproduct}
Let $I_{\abs{D}}$ be an elliptic divisor on a manifold $M$, and let $I_{\abs{D_1}},\ldots,I_{\abs{D_n}}$ be a choice of local normal crossing on some open $U$. Then $\mathcal{A}_{\abs{D}}$ is given by a repeated fiber product:
\begin{equation*}
\rest{\mathcal{A}_{\abs{D}}}{U} \simeq \mathcal{A}_{\abs{D_1}} \times_{TM} \times \cdots \times_{TM} \mathcal{A}_{\abs{D_n}}.
\end{equation*}
\end{lemma}

Above we argued that the ideal $I_{|D|}$ determines the elliptic tangent bundle. The converse is also true. Namely, suppose we are given any Lie algebroid $L \to M$ whose rank agrees with the dimension of $M$. The anchor map, $\rho$, induces a bundle map $\det_\rho\colon \wedge^n L\to \wedge^n TM$, which can be regarded as a section of the real line bundle $\wedge^n L^* \otimes \wedge^n TM$. That is, $L$ determines the real divisor $(\wedge^n L^* \otimes \wedge^n TM, \det_\rho)$. Given the local expression for generators of $\elli$ we have:

\begin{lemma}\label{lem:ideal}
The elliptic tangent bundle determines its underlying ideal.
\end{lemma}

When an elliptic divisor is induced from a complex log divisor we can relate the Lie algebroids.

\begin{proposition}\label{prop:ellicommute}
Let $I_D$ be a complex log divisor, and let $I_{\abs{D}}$ be the induced elliptic divisor. Then
\begin{equation*}
\mathcal{A}_D \times_{T_{\cc}M} \mathcal{A}_{\bar{D}} = \elli \otimes \cc.
\end{equation*}
\end{proposition}
\begin{proof}
Using local coordinates as in Lemma \ref{lem:loccomplexlogiso}, we see that the anchors of $\mathcal{A}_D$ and $\mathcal{A}_{\bar{D}}$ are transverse and hence we can form their fiber product. The anchor maps on sections, mapping to $\Gamma(T_\mathbb{C} M)$, satisfy
\begin{align*}
\rho(\Gamma(\mathcal{A}_D \times_{T_{\cc}M} \mathcal{A}_{\bar{D}})) = \rho(\Gamma(\mathcal{A}_D)) \cap_{T_{\cc}M} \rho(\Gamma(\mathcal{A}_{\bar{D}})).
\end{align*}
The left-hand side consists of the vector fields preserving $I_D \otimes \bar{I}_{D} = I_{\abs{D}}\otimes \cc$. Therefore $\mathcal{A}_D \times_{T_{\cc}M} \mathcal{A}_{\bar{D}}$ is isomorphic to the complexification of the Lie algebroid from Definition~\ref{def:ellialg}.
\end{proof}
Our next step is to compute the cohomology of the complex log tangent bundle. To do that we need the following topological result regarding divisors with normal crossings.
\begin{lemma}\label{lem:homotopytype}
If $D$ is the standard normal crossing complex log (or elliptic) divisor with intersection number $n$, then $C(i)$ is homotopic to the disjoint union of $\binom{n}{i}$ copies of $T^{n-i}$.
\end{lemma}
\begin{proof}
Let $I_{D_1},\ldots,I_{D_n}$ be the complex log divisors corresponding to $z_1,\ldots,z_n$ respectively. Then
\begin{align*}
C(i) = D(i)\backslash D(i+1) = ~~~~\bigsqcup_{(j_1,\ldots,j_i) \subset (1,\ldots,n)} (D_{j_1}\cap \cdots \cap D_{j_i})\backslash D(i+1).
\end{align*}
All the components in $C(i)$ are diffeomorphic. For notational clarity we will thus consider the component corresponding to the ordered multi-index $I := (1,2,\ldots,i)$. We have
\begin{align*}
D_I := D_1 \cap \cdots \cap D_i = \underbrace{\set{0} \times \cdots \times \set{0}}_{i-\text{times}} \times \underbrace{\cc \times \cdots \times \cc}_{(n-i)-\text{times}}.
\end{align*}
To proceed we consider the intersection of $D_I$ with $D(i+1)$. To write this down, for $k \in \mathbb{N}$ let
\begin{equation*}
		\textbf{0}_k := \underbrace{\set{0} \times \cdots \times \set{0}}_{k-\text{times}}, \qquad\qquad \cc_k := \underbrace{\cc \times \cdots \times \cc}_{k-\text{times}}.
\end{equation*}
With this notation in hand, one readily verifies that
\begin{equation*}
D_I \cap D(i+1) = \textbf{0}_i \times \set{0} \times \cc_{n-i-1} \cup \textbf{0}_i \times \cc \times \set{0} \times \cc_{n-i-2} \cup \cdots \cup \textbf{0}_i \times \cc_{n-i-1} \times \set{0}.
\end{equation*}
From this we immediately see that
\begin{equation*}
D_I \backslash  D(i+1) = \underbrace{\set{0} \times \cdots \times \set{0}}_{i-\text{times}} \times \underbrace{\cc^* \times \cdots \times \cc^*}_{(n-i)-\text{times}},
\end{equation*}
which is homotopic to $T^{n-i}$. Therefore, all components of $C(i)$ are homotopic to $T^{n-i}$, and as there are $\binom{n}{i}$ of these this finishes the proof.
\end{proof}
The following result regarding cohomology is the self-crossing analogue of \cite[Theorem 1.3]{CG17}, and appears in \cite{Del71} in the algebraic context.
\begin{theorem}\label{thm:complexcom}
Let $D = (L,\sigma)$ be a complex log divisor on a manifold $M$. Then the inclusion $\iota\colon M\backslash D \hookrightarrow M$ induces an isomorphism
\begin{align*}
H^k(M,\mathcal{A}_D) \simeq H^k(M\backslash D,\cc).
\end{align*}
\end{theorem}
\begin{proof}
We give an argument in the same spirit as \cite{Del71}. As is shown there it suffices to show that $\iota$ induces an isomorphism on the level of sheaf cohomology. Below we will implicitly identify the sheaf $\Omega^{\bullet}(M\backslash D)$ with its push-forward $\iota_*(\Omega^{\bullet}(M\backslash D))$. Given a point $p \in M \backslash D$ and a contractible neighbourhood $U$ of $p$ which is disjoint from $D$ we have that $\mathcal{A}_D = TD$ and hence $H^k(U,\mathcal{A}_D) = H^k(U\backslash D)$. Next, for any $j$ less than or equal to the intersection number of $D$, let $p \in C(j)$. Let $z_1,\ldots,z_j$ be coordinates on an open neighbourhood $U$ of $p$ as in Lemma \ref{lem:loccomplexlogiso}. In these coordinates we see that $H^\bullet(U,\mathcal{A}_D)$ is the free algebra generated by $\set{1,d\log z_1,\ldots, d\log z_j}$. By Lemma \ref{lem:homotopytype} we have that $U\backslash D$ is homotopic to $\mathbb{T}^j$. Under these identifications, the cochain morphism $\iota^*$ takes the generators of $H^\bullet(U,\mathcal{A}_D)$ to the generators of $H^{\bullet}(U\backslash D)$. Therefore we conclude that $\iota^*$ is a local isomorphism and hence also globally, which finishes the proof.
\end{proof}
\subsection{Residue maps}\label{sec:residues}
Let $(I_{|D|},\mathfrak{o})$ be a co-oriented smooth elliptic divisor on a manifold $M$. The restriction of the smooth elliptic tangent bundle to $D$ fits into a sequence of Lie algebroids
\begin{equation*}
0 \rightarrow \rest{\ker \rho}{D} \rightarrow \rest{\Ad}{D} \rightarrow TD \rightarrow 0.
\end{equation*}
In \cite{CG17} it is explained that this sequence induces a cochain map 
\begin{equation*}
\Res_q\colon \Omega^{\bullet}(\elli) \rightarrow \Omega^{\bullet-2}(TD),
\end{equation*}
called the elliptic residue. In local Morse--Bott coordinates for the divisor, the elliptic residue map is given by
\begin{equation*}
\Res_q(\alpha) = \iota^*_D(\iota_{r \partial_r} \iota_{\partial_\theta} \alpha), \qquad \alpha \in \Omega^\bullet(\mathcal{A}_{|D|}).
\end{equation*}
\begin{definition}
Let $(I_{\abs{D}},\frak{o})$ be a co-oriented elliptic divisor, and let $\alpha \in \Omega^\bullet(\elli)$. Its \textbf{elliptic residue} is defined by 
\begin{equation*}
\Res_q(\alpha) := \Res_q(\iota^*_{M\backslash D(2)}\alpha),
\end{equation*}
where the right-hand side is the elliptic residue for the smooth elliptic divisor $\rest{I_{\abs{D}}}{M\setminus D(2)}$.
\end{definition}
There are more residue maps associated to self-crossing elliptic divisors, but for the purpose of the present paper we will only briefly mention the ones we need.
\begin{definition}
Let $(I_{\abs{D}},\frak{o})$ be a co-oriented elliptic divisor and let $\omega \in \Omega^2(\elli)$. Consider oriented coordinates around a point $p \in D(k)$ with $k \geq 2$ as in Lemma~\ref{lem:locMB}, and define three types of residues:
\begin{equation*}
\Res_{r_ir_j}\omega(p) := \omega_p(r_i\partial_{r_i},r_j\partial_{r_j}),~ \Res_{r_i\theta_j}\omega(p) := \omega_p(r_i\partial_{r_i},\partial_{\theta_j}),~ \Res_{\theta_i\theta_j}\omega(p) := \omega_p(\partial_{\theta_i},\partial_{\theta_j}). \qedhere
\end{equation*}
\end{definition}
These expressions do not depend on the chosen coordinates. They do depend on the choice of co-orientation and the particular ordering of the divisors, but only up to signs.

\section{Self-crossing stable generalized complex structures}\label{sec:generalizedcomplex}
In this section we will start our discussion of generalized complex geometry, and self-crossing stable generalized complex structures in particular. We will first recall the basics of generalized complex geometry, before defining the (self-crossing) stable condition, using the divisors of Section~\ref{sec:divisors}. We then show how these can be described using symplectic-like forms in the complex log tangent bundle (Theorem~\ref{th:correspondence}), in analogy with \cite{CG17}. Next we discuss that these are in fact symplectic structures in its associated elliptic tangent bundle, satisfying certain additional cohomological conditions (Theorem~\ref{th:correspondence2}). From that point onwards we can proceed to study self-crossing stable generalized complex structures using symplectic techniques.
\subsection{Generalized complex structures}
Generalized geometry refers to the study of geometric structures on $\mathbb{T}M:= TM \oplus T^*M$, for a manifold $M$. We briefly recall the notions from generalized complex geometry which are needed in this paper. For a more in depth discussion see \cite{Gua07}. 
\begin{definition}
A \textbf{generalized complex structure} on a manifold $M$ is a pair $(\mathbb{J},H)$, where $H\in \Omega^3(M) $ is a closed three-form and $\mathbb{J}$ is an endomorphism of $\mathbb{T}M$ for which $\mathbb{J}^2 = -\text{Id}$ and the $+i$-eigenbundle $L \subset (\mathbb{T}M)\otimes \cc$ is involutive with respect to the Dorfman bracket:
\begin{equation*}
[\![X+\xi,Y+\eta]\!]_H := [X,Y] + \mathcal{L}_X\eta - \iota_Yd\xi + \iota_X\iota_Y H, \qquad X+\xi, Y+\eta \in \Gamma(\mathbb{T}M).\qedhere
\end{equation*}
\end{definition}
Two generalized complex structures $(\mathbb{J},H)$ and  $(\mathbb{J}',H')$ are \textbf{gauge equivalent} if there exists $B \in \Omega^2(M)$ such that $H' = H + dB$ and, using the associated map $B^\flat\colon TM \to T^*M$, we have
\begin{equation*}
	\mathbb{J}' = \left(\begin{smallmatrix}
		1 & B^{\flat} \\
		0 & 0
	\end{smallmatrix}
	\right)
	\mathbb{J}
	\left(
	\begin{smallmatrix}
		1 & -B^{\flat} \\
		0 & 0
	\end{smallmatrix}
	\right).
\end{equation*}
Given an element $X + \xi \in \mathbb{T}M$, let $(X+\xi) \cdot \rho := \iota_X \rho + \xi \wedge \rho$ denote the Clifford action of $\mathbb{T}M$ on elements $
\rho \in \wedge^{\bullet}T^*M$, which are called \textbf{spinors}. Moreover, define transposition of an element in $\Gamma(\wedge^{\bullet}T^*M)$ on decomposable degree $k$-forms by
\begin{equation*}
(\alpha_1\wedge \cdots \wedge \alpha_k)^T := \alpha_k \wedge \cdots \wedge \alpha_1, \qquad \alpha_i \in \Omega^1(M).
\end{equation*}
The \textbf{Chevalley pairing} on spinors, $(\cdot,\cdot)_{\text{Ch}}\colon \wedge^\bullet T^*M \times \wedge^\bullet T^*M \to \wedge^{\text{top}} T^* M$, is defined as
\begin{equation}
(\gamma,\rho)_{\text{Ch}} := (\gamma \wedge \rho^T)_{\text{top}}, \quad \gamma,\rho \in \Gamma(\wedge^{\bullet}T^*M).
\end{equation}
Generalized complex structures can be equivalently described using the following:
\begin{lemma}[\cite{Gua07}]
There is a one-to-one correspondence between generalized complex structures $(\mathbb{J},H)$ and complex line subbundles $K \subset \wedge^{\bullet}T^*_{\cc}M$, satisfying the following properties:
\begin{itemize}
\item For all $x \in M$ the vector space $K_x$ is generated over $\cc$ by spinors of the form
\begin{equation*}
\left(e^{B+i\omega}\wedge \Omega\right)_x, \quad B,\omega \in \Omega^2(M), \quad \Omega \in \Omega^k(M),
\end{equation*}
where $\Omega$ is a decomposable form;
\item For every nonvanishing local section $\rho \in \Gamma(K)$, there exists $u \in \Gamma(\mathbb{T}_{\cc}M)$ such that $d \rho + H \wedge \rho = u \cdot \rho$;
\item For all non-zero $\rho_x \in K_x$, we have $(\rho_x,\bar{\rho}_x)_{\text{Ch}} \neq 0$.
\end{itemize}
\end{lemma}
The line bundle $K$ is called the \textbf{canonical line bundle} of $\mathbb{J}$. It can be defined in terms of the generalized complex structure by the relation
\begin{equation*}
	L = \set{u \in \mathbb{T}_{\cc}M : u \cdot K = 0},
\end{equation*}
where $L$ is the $+i$-eigenbundle of $\mathbb{J}$. Note here that $\mathbb{T}_\cc M = \mathbb{T}M \otimes \cc$.
\begin{definition}
Let $(\mathbb{J},H)$ be a generalized complex structure, and let $K$ be its canonical line bundle. The map $s\colon K \rightarrow \wedge^0T^*_{\cc}M = \cc$ defined by $\rho \mapsto \rho_0$, sending a spinor to its degree-zero part, defines a section $s \in \Gamma(K^*)$ called the \textbf{anticanonical section} of $\mathbb{J}$.
\end{definition}

\begin{example}
Given a complex structure, $J$, or a symplectic structure, $\omega$, on a manifold $M$, we can endow $M$ with a generalized complex structure (for $H=0$) given by
\begin{align*}
\mathbb{J}_J := \begin{pmatrix}
-J & 0 \\
0 & J^*
\end{pmatrix}, \qquad \mathbb{J}_\omega := \begin{pmatrix}
0 & -(\omega^\flat)^{-1} \\
\omega^\flat & 0
\end{pmatrix}.
\end{align*}

More interestingly, let $(M,J,\pi)$ be a holomorphic Poisson manifold with $\pi = \pi_R + i \pi_I$ and denote by $\pi_I^\sharp\colon T^*M \to TM$ the map associated to $\pi_I$. Then
\[ \mathbb{J}_{J,\pi} :=
\begin{pmatrix}
-J &4 \pi_I^{\sharp} \\
0 & J^*
\end{pmatrix}
\]
is also a generalized complex structure.

The corresponding canonical line bundles and anti-canonical sections are given by:
\begin{align*}
K_{\omega} &= \inp{e^{i\omega}},  &&K_J = \wedge^{n,0}T^*M,   &&K_{J,\pi} = e^{\pi}(\wedge^{n,0}T^*M),\\ 
s_{\omega} &\equiv 1,&&~ s_J \equiv 0, && ~s_{J,\pi} = \iota_{\pi^n}\Omega_{\text{vol}}.\qedhere
\end{align*}
\end{example}
There is an interesting relation between generalized complex geometry and Poisson geometry obtained in \cite{MR2861779}. Given a generalized complex structure $\mathbb{J}$, denote by $\pi_{\mathbb{J}}^\sharp\colon T^*M \to TM$ the associated bundle map obtained by $\pi_{\mathbb{J}}^\sharp = {\rm pr}_{TM} \circ \mathbb{J}|_{T^*M}$. This map is skew-symmetric, and:
\begin{lemma}\label{lem:poisgen}
Let $\mathbb{J}$ be a generalized complex structure. Then $\pi_{\mathbb{J}} \in \mathfrak{X}^2(M)$ is a Poisson structure. Moreover, if $\mathbb{J},\mathbb{J}'$ are gauge-equivalent generalized complex structures, then $\pi_{\mathbb{J}} = \pi_{\mathbb{J}'}$.
\end{lemma}

\subsection{Stable generalized complex structures}
In this section we will extend the notion of stable generalized complex structures to allow for the degeneracy locus to have self-crossing singularities. The main reason to allow for normal crossing singularities is that it gives much more flexibility: for example, this class is now closed under taking products. This makes it easier to provide examples (see Section~\ref{sec:examples}), and there are also more constructions available, such as the connect sum procedure of Section~\ref{sec:connectedsum}. Moreover, in four dimensions these structures can be used to construct smooth stable generalized complex structures, as explained in Section~\ref{sec:4d}. 
\begin{definition}\label{defn:scsgcs}
A \textbf{(self-crossing) stable generalized complex structure} is a generalized complex structure such that its anticanonical divisor $D = (K^*,s)$ defines a complex log divisor. We call it \textbf{smooth stable} when $D = (K^*,s)$ defines a smooth complex log divisor.
\end{definition}
We can immediately see that the class of such structures is closed under products.
\begin{example}
Let $(L_i,\sigma_i)$ be two complex log divisors with intersection numbers $n_i$ on two manifolds $M_i$ for $i= 1,2$. Then the pair $(\pi_1^*L_1 \otimes \pi_2^*L_2,\pi_1^*\sigma_1 \otimes \pi_2^*\sigma_2)$ on the product manifold $M = M_1 \times M_2$, with projection maps $p_i\colon M \to M_i$, is a complex log divisor with intersection number $n_1 + n_2$. This shows that the product $(M,\pi_1^* H_1 + \pi_2^* H_2, \mathbb{J}_1 \oplus \mathbb{J}_2)$ of two stable generalized complex manifolds $(M_i,H_i,\mathbb{J}_i)$ is endowed with a stable generalized complex structure.
\end{example}
Stable generalized complex structures can come about via holomorphic Poisson structures.
\begin{example}\label{exa:holpoissonsgcs}
Let $\pi \in \mathfrak{X}^2(M^{2n};\cc)$ be a holomorphic Poisson structure such that the pair $(\wedge^{n,0} T M,\wedge^n\pi)$ is a complex log divisor. Then $\mathbb{J}_{J,\pi}$ is a stable generalized complex structure.
\end{example}
An important gain from allowing self-crossings is that deformations of higher dimensional Fano manifolds by holomorphic Poisson bivectors may provide examples of these structures, but are never smoothly stable \cite[Theorem 26]{GP13}.

\begin{example}\label{th:CP2n}
One can readily construct a holomorphic Poisson structure on $\cc P^{2n}$  for which $(\wedge^nTM,\wedge^n\pi)$ is a complex log divisor, making $\cc P^{2n}$ into a stable generalized complex manifold. Indeed, using the standard coordinates $(z_1,\dots,z_{2n+1})$ on $\cc^{2n+1}$, let
\begin{align*}
\tilde{\pi} := z_1z_2 \partial_{z_1}\wedge \partial_{z_2} + \cdots z_{2n-1}z_{2n}\partial_{z_{2n-1}}\wedge \partial_{z_{2n}}
\end{align*}
be a holomorphic Poisson structure on $\cc^{2n+1}$. This bivector is scaling invariant and therefore descends to a bivector $\pi$ on $\cc P^{2n}$. By naturality of the Schouten bracket it follows that $\pi$ is Poisson, and by direct computation we see that $(\wedge^nT_{\cc}M,\wedge^n\pi)$ is a complex log divisor. 
\end{example}

\subsection{Complex log symplectic structures}\label{sec:cplx-log}
The next sections aim to prove that stable generalized complex structures are equivalent to a certain type of symplectic structures on an associated elliptic tangent bundle. Before we do so, we first prove, in this section, that they are equivalent to an auxiliary structure. This will be a symplectic-like structure for the complex log tangent bundle.

Given a complex log divisor $I_D$ and its induced elliptic divisor $I_{\abs{D}}$ we consider the Lie algebroid morphism $\iota\colon \elli \otimes \cc \rightarrow \mathcal{A}_D$ obtained from  Proposition \ref{prop:ellicommute}. If we compose the pullback $\iota^*$ with taking the imaginary part of a form we obtain a cochain morphism $\Im^*\colon \Omega^{\bullet}(\mathcal{A}_D) \rightarrow \Omega^{\bullet}(\elli)$.
\begin{definition}\label{def:complexlog}
A form $\sigma \in \Omega^2(\mathcal{A}_D)$ is called \textbf{complex log symplectic} if $d\sigma = H \in \Omega^3(M;\rr)$ and $\Im^*\sigma \in \Omega^2(\elli)$ is non-degenerate. Two complex log symplectic forms $\sigma,\sigma' \in \Omega^2(\mathcal{A}_D)$ are said to be \textbf{gauge equivalent} if there exists a two-form $B \in \Omega^2(M,\rr)$ such that $\sigma' = \sigma + B$. 
\end{definition}
Since the symplectic structure given by the imaginary part of a complex log symplectic form clearly contains important information, it is useful to give it a name as well. 
\begin{definition} Let $M$ be a manifold with an elliptic divisor $I_{|D|}$. A \textbf{(self-crossing) elliptic symplectic form} is an elliptic two-form $\omega \in \Omega^2(\elli)$ that is closed and non-degenerate.
\end{definition}

The following theorem is the self-crossing generalisation of \cite[Theorem 3.2]{CG17}.
\begin{theorem}\label{th:correspondence}
Let $M$ be a manifold. There is a one-to-one correspondence between stable generalized complex structures with self-crossings on $M$ and isomorphisms classes of complex log divisors endowed with complex log symplectic forms with self-crossings. Moreover, this correspondence preserves gauge equivalences.

Explicitly, the correspondence is given by the map
\[
\left\{
\begin{array}{c}
(\mathbb{J}, H)\colon\\
\mathbb{J} \mbox{  is a stable GCS}
\end{array}
\right\}
 \to
 \left\{
\begin{array}{c}
(I_{D}, \sigma)\colon\\
I_{D} \mbox{ is a complex log divisor and }\\
\sigma \in \Omega(\cplx)\mbox{ is a complex log-symplectic form.}
\end{array}
\right\}
\]
where $I_D$ is the divisor induced by the anticanonical section, and $\sigma = \rho_2/\rho_0$ where $\rho$ is any local spinor for $\mathbb{J}$. 
\end{theorem}
\begin{proof} 
We will first consider the direct implication. Let $\rho \in \Omega^{\bullet}(M)$ be a local pure spinor for the stable generalized complex structure, defined on an open set $U\subset M$. On $U\backslash D$ the anticanonical section is non-vanishing, and therefore we have that $\rho = ce^{\varepsilon}$ on $U\backslash D$ for some $c \in C^{\infty}(M;\cc)$ and $\varepsilon \in \Omega^2(M;\cc)$. Looking at the degree-zero part of this equation we obtain $c = \rho_0$, and looking at the degree-two part we obtain $\varepsilon = \rho_2/\rho_0$. By continuity we conclude that $\rho = \rho_0 e^{\rho_2/\rho_0}$ on the entirety of $U$.
\begin{claim*}
The two-form $\sigma := \rho_2/\rho_0$ defines a global smooth complex log form, with $d\sigma = -H$.
\end{claim*}
\begin{proof}[Proof of claim]
Assume that the local form, $\sigma$, defined using a local canonical section is indeed a local complex log form and let $ \tilde\sigma$ be another local form obtained from another local canonical section, $\tilde\rho$. Then, from the previous argument we have $\tilde\rho = \tilde\rho_0 e^{\tilde\rho_2/\tilde\rho_0}$. In particular we see that as complex log forms we have
\[e^{\sigma} = \frac{\rho}{\rho_0}= \frac{\tilde\rho}{\tilde\rho_0} = e^{\tilde\sigma},\]
where the middle equality follows from the fact that $ \frac{\rho}{\rho_0}$ and $ \frac{\tilde\rho}{\tilde\rho_0} $ are sections of the canonical bundle with the same degree-zero component. Therefore $\sigma$ is a global complex log form.

Hence to conclude the result we must prove that $\sigma$ is a local log form. By the integrability of the generalized complex structure there exist $X  + \xi \in \Gamma(\mathbb{T}M)$ such that
\begin{align*}
d\rho_0 &= \iota_X\rho_2 + \xi \wedge \rho_0,\\
d\rho_2 &= \iota_X\rho_4 + \xi \wedge \rho_2 - \rho_0H.
\end{align*}
Because $\rho_4 = \frac{1}{2\rho_0}\rho_2\wedge \rho_2$ on $U\backslash D$, we find that
\begin{align*}
d\rho_2 = \frac{\iota_X(\rho_2\wedge \rho_2)}{2\rho_0} + \xi \wedge \rho_2 - \rho_0 H.
\end{align*}
On $U\backslash D$ we thus have:
\begin{align*}
d\left(\frac{\rho_2}{\rho_0}\right) &= \frac{d\rho_2}{\rho_0}-\frac{\rho_2\wedge d\rho_0}{\rho_0^2}\\
&= \frac{\iota_X\frac{\rho_2^2}{2\rho_0} + \xi \wedge \rho_2}{\rho_0} - H -\frac{\rho_2 \wedge (\iota_X\rho_2 + \xi \wedge \rho_0)}{\rho_0^2}\\
&= -H.
\end{align*}
Therefore $d(\frac{\rho_2}{\rho_0})$ extends to a smooth form on $U$. Rearranging the above terms a bit we find
\begin{align*}
\frac{\rho_2\wedge d\rho_0}{\rho_0} = d\rho_2 +\rho_0 H,
\end{align*}
hence the left-hand side extends to a smooth form on  $U$. 
Let $\rho_0^1,\ldots,\rho_0^n$ be functionally independent transverse-vanishing functions such that $\rho_0 = \prod_i \rho_0^i$. Then
\begin{align*}
d\rho_0 = \sum_{i=1}^n \rho_0^1 \cdots \widehat{\rho_0^i} \cdots \rho_0^n d\rho_0^i.
\end{align*}
We can write
\begin{align*}
\rho_2 = \sum_{i,j} f_{ij} d\rho_0^i \wedge d\rho_0^j + \sum_k d\rho_0^k \wedge \alpha_k + \beta,
\end{align*}
where $\alpha_k \in \Omega^1(M)$ and $\beta \in \Omega^2(M)$ do not contain $d\rho_0^i$-terms. We consider then that
\begin{align*}
d\rho_0 \wedge \rho_2 &= \sum_{i,j,\ell} f_{ij} \rho_0^1 \cdots \hat{\rho}_0^\ell \cdots \rho_0^n d\rho_0^i\wedge d\rho_0^j \wedge d\rho_0^l\\
&\qquad- \sum_{k,\ell} \rho_0^1 \cdots \hat{\rho}_0^\ell \cdots \rho_0^nd\rho_0^k \wedge d\rho_0^\ell \wedge \alpha_k + \sum_{\ell=1}^n \rho_0^1 \cdots \widehat{\rho_0^\ell} \cdots \rho_0^n d\rho_0^i \wedge \beta.
\end{align*}
Because we know that the left-hand side remains smooth after dividing by $\rho_0$ we find that
\begin{align*}
\sum_{i,j,\ell} \frac{f_{ij}}{\rho_0^\ell} d\rho_0^i\wedge d\rho_0^j \wedge d\rho_0^\ell - \sum_{k,\ell} d\rho_0^k \wedge d\rho_0^\ell \wedge \frac{\alpha_k}{\rho_0^\ell} + \sum_\ell \frac{d\rho_0^\ell}{\rho_0^\ell}\wedge \beta
\end{align*}
is smooth. Because every term in the above sum is linearly independent from the others, we conclude that each of the terms needs to be smooth separately. From the first term we thus conclude that $\frac{f_{ij}}{\rho_0^l}$ needs to be smooth for all $i,j,\ell$ pairwise distinct. Hence $f_{ij} = \rho_0^1 \cdots \hat{\rho}_0^i \cdots \hat{\rho}_0^j \cdots \rho_0^n\tilde{f}_{ij}$ for some $\tilde{f}_{ij} \in C^{\infty}(M;\cc)$. Similarly we find that $\alpha_k = \rho_0^1 \cdots \hat{\rho}_0^k \cdots \rho_0^n\tilde{\alpha}_k$ for some $\tilde{\alpha}_k \in \Omega^1(M;\cc)$ and lastly that $\beta = \rho_0 \tilde{\beta}$ for some $\tilde{\beta} \in \Omega^2(M;\cc)$. Therefore we conclude that
\begin{align*}
\frac{\rho_2}{\rho_0} = \rho_2 = \sum_{i,j} \tilde{f}_{ij} d\log\rho_0^i \wedge d\log\rho_0^j + \sum_k d\log\rho_0^k \wedge \tilde{\alpha}_k + \tilde{\beta},
\end{align*}
extends to a smooth logarithmic form over $D$, thus $\sigma = \frac{\rho_2}{\rho_0} \in \Omega^2(U;\log D)$.
\end{proof}
The algebraic condition $(\rho,\bar{\rho})_{\text{Ch}} \neq 0$ results in
\begin{align*}
\abs{\rho_0}^2(\sigma-\bar{\sigma})^n \neq 0.
\end{align*}
Therefore we conclude that the elliptic form $\sigma - \bar{\sigma}$ is non-degenerate, and hence $\sigma$ is a complex log symplectic form.\\

Next we consider the converse implication. Let $I_D$ denote the complex log divisor, and let $\sigma \in \Omega^2(\mathcal{A}_D)$ be a complex log symplectic form for this ideal. Let $\inp{e^\sigma}_{\cc} \subset \Omega^{\bullet}(\mathcal{A}_D)$ denote the complex line generated by $e^{\sigma}$. Then the product  $I_D \otimes \inp{e^\sigma}_{\cc} \subset \Omega^\bullet(\mathcal{A}_D)$ is  in fact  smooth, that is, $I_D \otimes \inp{e^\sigma}_{\cc} \subset \Omega^{\bullet}(M)$. We will show that this line bundle defines a stable generalized complex structure. In local coordinates as in Lemma \ref{lem:loccomplexlogiso} we have that $\rho = z_1\cdot\ldots\cdot z_k e^\sigma$ is a local trivialisation of the line bundle. Now
\begin{align*}
(\rho,\bar{\rho})_{\text{Ch}} = \bar{z_1}^2 \cdot\ldots\cdot \bar{z_k}^2(\Im^*\sigma)^n,
\end{align*}
defines a volume form because $\Im^*\sigma$ is a nondegenerate elliptic form with self-crossings. We are left to prove integrability of $\rho$. Since integrability is a closed condition, it is enough check it in  $M\backslash D$, but by construction in this region the structure is just a $B$-field transform of a symplectic structure.
\end{proof}

\subsection{Equivalence with elliptic symplectic}
As we just saw, a stable generalized complex structures is closely related to an elliptic symplectic form. Yet, the elliptic tangent bundle, $\elli$, which appears in the context of generalized complex structures arises from a complex log tangent bundle, $\mathcal{A}_D$, and the elliptic symplectic form arises as imaginary part of complex log symplectic form. Therefore our next step is to pinpoint precisely which elliptic symplectic forms arise in this way. That is we are interested in describing the image of `taking the imaginary part': $\Im^*\colon \Omega^{\bullet}(\mathcal{A}_D) \rightarrow \Omega^{\bullet}(\elli)$.

From now on we denote this image by $\Omega^{\bullet}_\Im(\elli) \subset  \Omega^{\bullet}(\elli)$.  Describing elements in $\Omega^{\bullet}_\Im(\elli)$ of arbitrary degree for a general complex divisor is a little involved, so we will focus on  our object of interest: two-forms. To describe elements in $\Omega^{2}_\Im(\elli)$ we need to use the residue maps  for points in $D(2)$ introduced in Section~\ref{sec:residues}, such as, $\Res_{r_ir_j}$ and $\Res_{r_i\theta_j}$. Recall that these residues depend on an ordering of the coordinates and a choice of co-orientation, the kernels of some combinations of these residues only depends on the co-orientation of the divisor. To be precise, given a co-oriented elliptic divisor $(I_{|D|},\frak{o})$, the spaces $\operatorname{ker}(\Res_q)$, $\operatorname{ker}(\Res_{\theta_ir_j} - \Res_{r_i\theta_j})$ and $\operatorname{ker}(\Res_{r_ir_j}+ \Res_{\theta_i\theta_j})$ do not depend on the order of the divisors.

\begin{lemma}\label{lem:degree2impart}
Let $I_D$ be a complex log divisor and let $(I_{|D|},\frak{o})$ be its associated co-oriented elliptic divisor. Then
\[\Omega_{\Im}^2(\elli) =  \operatorname{ker}(\Res_q)\cap(\operatorname{ker}(\Res_{\theta_ir_j} - \Res_{r_i\theta_j})) \cap (\operatorname{ker}(\Res_{r_ir_j}+ \Res_{\theta_i\theta_j})).\]
\end{lemma}
\begin{proof}
Choose local coordinates as in Lemma~\ref{lem:locMB}. For all pairs $i,j$ and $\beta \in \Omega^1(M;\rr)$ we have:
\begin{align*}
\Im^*(d\log z_i \wedge d\log z_j) &= d\theta_i \wedge d\log r_j + d\log r_i \wedge d\theta_j,\\
\Im^*(id\log z_i \wedge d\log z_j) &= d\log r_i \wedge d\log r_j - d\theta_i \wedge d\theta_j,\\
\Im^*(d\log z_i \wedge \beta) &= d\theta_i \wedge \beta,\\
\Im^*(id\log z_i \wedge \beta) &= d\log r_i \wedge \beta.
\end{align*}
Therefore we see that $\Omega^2_{\Im}(\elli)$ lies in the intersection of the kernels. Conversely, if $\alpha$ is in the intersection of the kernels then locally it must be a linear combination of the above forms together with smooth forms. Therefore by the above computation, we see that $\alpha \in \Omega_{\Im}^2(\elli)$ which concludes the proof.
\end{proof}

Since $\Im^*$ is a map of complexes, $\Omega_{\Im}^\bullet(\elli) \subset \Omega^\bullet(\elli)$ is a subcomplex and we can compute its cohomology, as we will do in Section~\ref{sec:ellcohomology}. Next we see that the imaginary part of a complex two-form determines it up to gauge equivalence:
\begin{proposition}\label{prop:sessc}
Let $M$ be a manifold with a complex log divisor $I_D$, and let $(I_{|D|},\frak{o})$ be its induced co-oriented elliptic divisor. Then the following sequence of cochain complexes is exact:
\begin{align*}
0 \rightarrow \Omega^{\bullet}(M;\rr) \rightarrow \Omega^{\bullet}(\mathcal{A}_D) \overset{\Im^*}{\rightarrow} \Omega^{\bullet}_{\Im}(\elli) \rightarrow 0.
\end{align*}
\end{proposition}
\begin{proof}
We can check exactness of the above sequence by checking exactness of the corresponding sheaf sequence. By Lemma~\ref{lem:loccomplexlogiso} we may assume we are in the setting of Example~\ref{ex:mainexample}. The map $\Im^*$ is surjective by definition, and it is clear that that $\Omega^{\bullet}(M;\rr)$ lies in the kernel of $\Im^*$. Therefore we are left to show that if a general complex log form has vanishing imaginary part, then it must be smooth. Consider coordinates $(z_1,\ldots,z_k, \bar{z}_1,\ldots,\bar{z}_k,x_i,\ldots,x_l)$ and for multi-indices $I = (i_1,\ldots,i_l)$ make use of the following shorthand notation:
\begin{align*}
z_I &:= z_{i_1}\cdots z_{i_l}, \quad (d\log z)_I := d\log z_{i_1}\wedge \cdots \wedge d\log z_{i_l},\\
(d\bar{z})_I &:= d\bar{z}_{i_1}\wedge \cdots \wedge d\bar{z}_{i_l}.
\end{align*} 
Using these, a general complex log form $\rho \in \Omega^\bullet(\mathcal{A}_D)$ may be written locally as
\begin{equation*}
\rho = \sum_{I,J,K} \alpha_{IJK}(d\log z)_I \wedge (d\bar{z})_J \wedge (dx)_K,
\end{equation*}
where the $\alpha_{IJK} \in C^{\infty}(M;\cc)$ are smooth, and the sum ranges over all multi-indices. The vanishing of the imaginary part of $\rho$ implies
\begin{equation*}
0 = \sum_{I,J,K} \alpha_{IJK}(d\log z)_I \wedge (d\bar{z})_J \wedge (dx)_K - \sum_{I,J,K} \bar{\alpha}_{IJK}(d\log \bar{z})_I \wedge (d z)_J \wedge (dx)_K,
\end{equation*}
which using $z_I (d\log z)_I = (dz)_I$ gives:
\begin{equation*}
0 = \sum_{I,J,K} \bar{z}_J\alpha_{IJK}(d\log z)_I \wedge (d\log \bar{z})_J \wedge (dx)_K - \sum_{I,J,K} z_J\bar{\alpha}_{IJK}(d\log \bar{z})_I \wedge (d \log z)_J \wedge (dx)_K.
\end{equation*}
By linear independence, each of the terms involving $(d\log z)_I \wedge (d\log \bar{z})_J$ and their conjugates needs to vanish independently. That is, for all multi-indices $I,J,K$ we have
\begin{equation*}
0 = (\bar{z}_J\alpha_{IJK}(d\log z)_I \wedge (d\log \bar{z})_J - z_I\bar{\alpha}_{JIK}(d\log \bar{z})_J \wedge (d \log z)_I)\wedge (dx)_K,
\end{equation*}
and hence
\begin{equation*}
(\bar{z}_J\alpha_{IJK} - (-1)^{\abs{I}\abs{J}}z_I\bar{\alpha}_{JIK}) = 0,
\end{equation*}
from this we conclude that $\alpha_{IJK}$ is divisible by $z_I$. This proves that $\rho$ is a smooth form.
\end{proof}
Putting Theorem~\ref{th:correspondence}, Proposition~\ref{prop:ellivscomplex} and Proposition~\ref{prop:sessc} together we have the following.

\begin{theorem}\label{th:correspondence2}
Let $M$ be a manifold. There is a correspondence between gauge equivalence classes of stable generalized complex structures with self-crossings on $M$ and isomorphism classes of co-oriented elliptic divisors, $(I_{|D|},\frak{o})$, endowed with an elliptic symplectic form $\omega \in \Omega_{\Im}^2(\elli)$.

Explicitly, this equivalence is induced by the map
\[
\left\{
\begin{array}{c}
(\mathbb{J}, H)\colon\\
\mathbb{J} \mbox{  is a Stable GCS}
\end{array}
\right\}
 \to
 \left\{
\begin{array}{c}
(I_{|D|}, \frak{o}, \omega)\colon\\
(I_{|D|}, \frak{o}) \mbox{ is a co-oriented elliptic divisor and }\\
\omega \in \Omega_{\Im}(\elli)\mbox{ is a symplectic form.}
\end{array}
\right\}
\]
which assigns to each stable generalized complex structure on $M$ the co-oriented elliptic divisor determined by its anticanonical section and the imaginary part of its corresponding complex log symplectic form.
\end{theorem}

\begin{remark}
Under the equivalence above, $[H] = \delta[\omega]$, where $\delta$ is the connecting morphism coming from Proposition~\ref{prop:sessc}.
\end{remark}

\subsection{Equivalence with nondegenerate elliptic Poisson} 
A consequence of Theorem \ref{th:correspondence2} is that nearly all the information of a stable generalized complex structure is already encoded in its underlying Poisson structure. 

\begin{theorem}
The gauge equivalence class of a stable generalized complex structure $(\mathbb{J},H)$ is fully determined by its underlying Poisson structure.
\end{theorem}
\begin{proof}
In the symplectic locus the Poisson structure is given by $\omega^{-1}$. By smooth continuation $\omega^{-1}$ on $M\backslash D$ determines $\omega$ on $\elli$, which in turn determines $\mathbb{J}$ up to gauge equivalence by Theorem \ref{th:correspondence2}. The only point that needs attention in this argument is that we extended $\omega^{-1}$ from $M\backslash D$ to $\elli$ but we did not argue yet that the Poisson structure itself determines $\elli$. This is indeed the case, as shown by Lemma~\ref{lem:4.18} below.\end{proof}

\begin{lemma}\label{lem:4.18}
Let $I_{|D|}$ be an elliptic divisor. Given $\omega \in \Omega^2(\elli)$ an elliptic symplectic form, let $\pi = \rho(\omega^{-1})$ be its associate Poisson bivector on $M$, where $\rho\colon \elli \to TM$ is the anchor map. Then $(\wedge^nTM,\wedge^n\pi)$ defines the elliptic divisor $I_{|D|}$.

Conversely, if $(\wedge^nTM,\wedge^n\pi)$ defines an elliptic divisor $I_{|D|}$, then $\pi$ admits a nondegenerate lift to $\elli$ and hence defines an elliptic symplectic structure.
\end{lemma}
\begin{proof}
In local coordinates  expressing $I_{|D|}$ as a normal crossing, $\omega^n$ is a volume form, hence there is a nonvanishing function, $f$, for which
\[
\omega^n = f d\log r_1\wedge  d\theta_1  \wedge  \dots \wedge d\log r_k\wedge  d\theta_k  \wedge dx_{2k+1}\wedge \dots dx_{2n}.
\]
Hence 
\[
\pi^n = \rho(\omega^{-n})= f^{-1} r_1\partial_{r_1} \wedge \partial_{\theta_1}  \wedge  \dots \wedge r_k\partial_{r_k} \wedge \partial_{\theta_k}   \wedge \partial_{x_{2k+1}}\wedge \dots \partial_{x_{2n}},
\]
which defines the divisor $I_{|D|}$.

Conversely, assume that $(\wedge^nTM,\wedge^n\pi)$ defines an elliptic divisor $I_{|D|}$. Due to \cite[Theorem A]{K18}, to prove that $\pi$ lifts to to $\elli$ it suffices to show that $\elli^*$ is locally generated by closed one-forms. From the local description in Lemma~\ref{lem:locMB} this is immediate.

Since, by Lemma~\ref{lem:ideal}, the ideal defined by $\elli$ is $I_{|D|}$, the lift above is non-degenerate, by \cite[Theorem A]{K18}.
\end{proof}
\subsection{Cohomology}\label{sec:ellcohomology}
Due to Theorem~\ref{th:correspondence2}, in order to study stable generalized complex structures, we can turn our attention to  symplectic forms in the associated elliptic tangent bundle that lie in the subcomplex given by $\Omega^\bullet_{\Im}(\elli)$. Therefore, the cohomology that is relevant to the study of these symplectic structures is not the elliptic cohomology but, $H^2_{\Im}(\elli)$, the cohomology of the subcomplex $\Omega^\bullet_{\Im}(\elli)$. We study this cohomology next.
\begin{proposition}\label{prop:cohimage}
Let $M$ be a manifold and let $I_{\abs{D}}$ be a co-oriented elliptic divisor. Then
\begin{align*}
H^i_{\Im}(\elli) \simeq H^i(M,M\backslash D) \oplus H^i(M\backslash D).
\end{align*}
\end{proposition}
\begin{proof}
We have the following morphism of cochain complexes
\[
\xymatrix{
0 \ar[r] & \Omega^{\bullet}(M,\rr) \ar[r] \ar[d] & \Omega^{\bullet}(\mathcal{A}_D) \ar[r] \ar[d] & \Omega^{\bullet}_{\Im}(\elli) \ar[r] \ar[d] & 0\\
0 \ar[r] & \Omega^{\bullet}(M,\rr) \ar[r]^(0.4){\iota^*_{M\backslash D}} & \Omega^{\bullet}(M\backslash D,\cc) \ar[r] & \Omega^{\bullet}(M \backslash D,\cc)/\Omega^{\bullet}(M,\rr) \ar[r] & 0
}
\]
where $\iota^*_{M\backslash D}$ is the natural inclusion, the middle vertical arrow  corresponds to restriction to $M\backslash D$ and the rightmost vertical arrow restriction composed with the quotient map. 

Therefore, we have the corresponding commutative diagram in cohomology:
\[
\xymatrix@R=18pt@C=12pt{
 \dots\ar[r] &H^\bullet(M) \ar[r] \ar[d] & H^\bullet(\mathcal{A}_D) \ar[r] \ar[d] & H^\bullet_{\Im}(\mathcal{A}_D) \ar[r] \ar[d] & H^{\bullet+1}(M) \ar[r] \ar[d] & H^{\bullet+1}(\mathcal{A}_D) \ar[d]\ar[r]& \dots\\
\dots\ar[r] & H^\bullet(M) \ar[r] & H^\bullet(M\backslash D,\cc) \ar[r] & C^\bullet \ar[r] & H^{\bullet+1}(M) \ar[r] & H^{\bullet+1}(M\backslash D,\cc) \ar[r]& \dots
}
\]
Here we let $C^\bullet$ denote the cohomology of the quotient complex. By Theorem \ref{thm:complexcom} and the Five Lemma we conclude that $H^\bullet_{\Im}(\elli) \simeq C^\bullet$. The quotient complex splits:
\begin{align*}
\Omega^{\bullet}(M\backslash D,\cc)/\Omega^{\bullet}(M,\rr) = \Omega^{\bullet}(M\backslash D,\rr)/\Omega^{\bullet}(M,\rr) \oplus i \Omega^{\bullet}(M\backslash D,\rr)
\end{align*}
from which we conclude that $C^\bullet \cong H^\bullet(M \backslash D) \oplus H^\bullet(M,M\backslash D)$. Here the last cohomology group is the relative cohomology of $M$ with respect to $M \backslash D$. Combining all of these isomorphisms we conclude that $H^\bullet_{\Im}(\elli) \simeq H^\bullet(M,M\backslash D) + i H^\bullet(M\backslash D) $ as desired.
\end{proof}

\section{Self-crossing stable structures in dimension four}\label{sec:4d}
In the remaining of this paper we will focus on stable generalized complex structures in four dimensions. From a practical point of view, dimension four is special because the way different components of the divisor intersect is very restricted, and at these intersections points the symplectic structure behaves as a meromorphic volume form. In this dimension it is particularly simple to state and prove a local normal form for a neighbourhood of a point (Theorem~\ref{thm:stablelocform}). Simple as this local form is, it has interesting consequences. Firstly, it can be used to show that every self-crossing stable generalized complex structure can be changed into a smooth stable structure (Theorem~\ref{th:smoothing}). Secondly it allows us to introduce a connected sum operation for stable  generalized complex structures (Theorem \ref{thm:ellipticglueing}).

\subsection{Normal forms}
Having related stable generalized complex structures to symplectic structures on a Lie algebroid, we have a wealth of symplectic techniques available to deal with them. One of the most basic tools from symplectic geometry,  the Moser Lemma, carries through to Lie algebroids with obvious modifications. This allows us to tackle deformations and corresponding neighbourhood theorems. As for regular symplectic structures, the (local) deformations are governed by the Lie algebroid cohomology in degree two. One new feature, however, is that the local cohomology of the elliptic tangent bundle is non-trivial and hence the local model must depend on parameters.

The next lemma is a direct adaptation of the Moser lemma.
\begin{lemma}
\label{lem:elliMoser}
Let $I_{\abs{D}}$ be an elliptic divisor on a manifold $M$, and let $\omega_1,\omega_2 \in \Omega^2_{\Im}(\elli)$ be two elliptic symplectic forms defined on a neighbourhood $U$ of a component of $D(i)$. If $[\omega_1] = [\omega_2] \in H^2_{\Im}(\elli)$ and $t\omega_1 + (1-t)\omega_2$ is non-degenerate for all $t\in [0,1]$, then there exists neighbourhoods $U_1, U_2 \subset U$ of $D(i)$ and a Lie algebroid isomorphism $\tilde\phi\colon \rest{\elli}{U_2} \rightarrow \rest{\elli}{U_1}$ such that $\tilde\phi^*\omega_1 = \omega_2$.
\end{lemma}
\begin{proof}
The proof is nearly identical to the usual proof of the Moser Lemma: let $\alpha\in \Omega^1_{\Im}(\elli)$ be such that  $\omega_1 - \omega_2= d\alpha$ and define  $X_t \in \Gamma(\elli)$ by $(t\omega_1 + (1-t)\omega_2)(X_t) = \alpha$. Then the time-one flow, $\tilde\varphi$, of $X_t$ on $\elli$ (over the flow, $\varphi$, of $\rho(X_t)$ on $U$) will satisfy $\tilde\phi^*\omega_1 = \omega_2$. Since $X_t \in \Gamma(\elli)$, $\rho(X_t)$ is tangent to $D(i)$ and hence if we take a possibly smaller neighbourhood $U_2$ of $D(i)$ we have $\varphi(U_2) \subset U$ and taking $U_1 = \varphi(U_2)$ we obtain the result.
\end{proof}
A Darboux Theorem for points in $D(1)$, the smooth locus of the divisor, was already obtained in \cite{CG17}. In four dimensions the only extra stratum available is $D(2)$ and we present the version of the Darboux Theorem for those points now.
\begin{proposition}\label{prop:normformelliptic}
Let $(I_{\abs{D}},\frak{o})$ be a co-oriented elliptic divisor on $M^4$. Let $\omega \in \Omega_{\Im}^2(M)$ be an elliptic symplectic form and let $p \in D(2)$. Choose coordinates that express $I_{|D|}$ as the standard elliptic divisor in a neighbourhood of $p$ and denote its residues by
\begin{equation*}
\lambda_1 = \Res_{r_1\theta_2}\omega(p)= \Res_{\theta_1r_2}\omega(p), \qquad 	\lambda_2 = \Res_{r_1r_2}\omega(p) = - \Res_{\theta_1\theta_2}\omega(p).
\end{equation*}
Then $\lambda_1^2 + \lambda_2^2 \neq 0$ and there exists coordinates $(r_1,\theta_1,r_2,\theta_2)$ on a neighbourhood of $p$ on which $I_{|D|}$ is the standard elliptic divisor and $\omega$ is given by:
\[
 \lambda_1 (d\log r_1 \wedge d\theta_2 +  d\log \theta_1 \wedge d\log r_2) + 
 \lambda_2 (d\log r_1 \wedge d\log r_2 - d\theta_1 \wedge d\theta_2).
\]
\end{proposition}
\begin{proof}
Indeed, $\{d\log r_1, d\theta_1,d\log r_2,d\theta_2\}$ is a frame for $\Omega^1(\elli)$  in a neighbourhood of $p$. Hence we can write $\omega$ in terms of wedge products of these generators with functions as coefficients. At $p$ all these coefficients are residues and due to Lemma \ref{lem:degree2impart} we have
\[
\omega(p) =  \lambda_1 (d\log r_1 \wedge d\theta_2 +  d\log \theta_1 \wedge d\log r_2) +  \lambda_2 (d\log r_1 \wedge d\log r_2 - d\theta_1 \wedge d\theta_2).
\]
Since $\omega^2(p) \neq 0$, we have $\lambda_1^2 + \lambda_2^2 \neq 0$.

Now consider
\[
\omega_0 = \lambda_1 (d\log r_1 \wedge d\theta_2 +  d\log \theta_1 \wedge d\log r_2) +  \lambda_2 (d\log r_1 \wedge d\log r_2 - d\theta_1 \wedge d\theta_2)  \in \Omega_{\Im}^2(\elli).
\]
Since $\omega(p) = \omega_0(p)$, the convex combination $t\omega + (1-t)\omega_0$ is symplectic in a neighbourhood of $p$  for all $t \in [0,1]$ and Proposition \ref{prop:cohimage} implies that $[\omega] = [\omega_0] \in H_{\Im}^2(\elli)$. Since $\{p\}$ is a connected component of $D(2)$, the Moser Lemma gives us the desired diffeomorphism between $\omega$ and $\omega_0$.
\end{proof}

A direct consequence is a normal form for stable generalized complex structures.
\begin{theorem}\label{thm:stablelocform}
Let $M^4$ be a stable generalized complex manifold. Then for every point $p \in D(2)$ there exists complex coordinates $(z_1,z_2)$ around $p$ where the complex log divisor is the standard one and such that a local trivialisation of the canonical line bundle is given by the pure spinor
\begin{align*}
\rho = e^B(\lambda z_1 z_2 + dz_1 \wedge dz_2),
\end{align*}
for some $\lambda \in \cc$ and $B \in \Omega^2(M;\rr)$.
\end{theorem}
\begin{proof}
By Theorem \ref{th:correspondence2}, gauge equivalence classes of stable generalized complex structures are in equivalence with elliptic symplectic forms $\omega\in \Omega^2_{\Im}(M)$. By Proposition \ref{prop:normformelliptic}, any such symplectic form is given, in appropriate coordinates, by
\begin{align*}
\omega & = \lambda_1 (d\log r_1 \wedge d\theta_2 +  d\log \theta_1 \wedge d\log r_2) +  \lambda_2 (d\log r_1 \wedge d\log r_2 - d\theta_1 \wedge d\theta_2) \\
& =  \Im^*((\lambda_1 + i \lambda_2)(d\log r_1 + i d \theta_1) \wedge (d\log r_2 + i d \theta_2)).
\end{align*}
Hence, up to the action of 2-forms, the stable generalized complex structure is given by
\[
(\lambda_1 + i \lambda_2)z_1 z_2 + dz_1 dz_2,
\]
with $z_j = r_j e^{i\theta_j}$ and the stated normal form follows.
\end{proof}
\subsection{Locally complex elliptic symplectic structures}
Theorem \ref{th:correspondence2} gives an equivalence between stable generalized complex structures and certain elliptic symplectic forms together with a co-orientation of the corresponding elliptic divisor. Of course, the main use of that result is to work on the symplectic side to conclude properties of the generalized complex structure. Here there is a minor difficulty: we must co-orient an elliptic divisor in the hopes of getting the desired residue relations, but there is no preferred way to do that step. That is, if we are interested in constructing a generalized complex structure on a given manifold $M$ with an elliptic divisor $I_{|D|}$,  we must choose a co-orientation for $D$ but must also  keep in mind that we may have started with the wrong choice.  We find it fruitful to  introduce a notion that is independent of the choice of co-orientation and will allow us to get a better grip on the problem.
\begin{definition}\label{def:locallycomplex}
Let $I_{\abs{D}}$ be an elliptic divisor, and let $\omega \in \Omega^2(\elli)$ be an elliptic symplectic form. We say that $\omega$ is \textbf{locally complex} if around every point there exists an open neighbourhood $U$ and a complex log divisor $I_D$ on $U$ inducing the restricted divisor $I_{\abs{D}}|_U$ together with a complex log symplectic form $\sigma \in \Omega^2(U;\mathcal{A}_D)$ such that $\Im^*\sigma = \omega$.
\end{definition}

Concretely, $\omega$ is locally complex if and only if its elliptic residue vanishes over $D(1)$ and over $D(2)$, for a choice of co-orientation of $D$, one of the following two possibilities holds:
\begin{align}
 (\Res_{\theta_ir_j}\omega - \Res_{r_i\theta_j} \omega = 0\qquad &\mbox{and}\qquad \Res_{r_ir_j}\omega+ \Res_{\theta_i\theta_j}\omega =0),\qquad \mbox{or} \label{eq:loc complex 1}\\
 (\Res_{\theta_ir_j}\omega + \Res_{r_i\theta_j} \omega = 0\qquad &\mbox{and}\qquad \Res_{r_ir_j}\omega- \Res_{\theta_i\theta_j}\omega =0),\label{eq:loc complex 2}
 \end{align}
with the second possibility indicating that the chosen co-orientation for one of the components of $D$ is not compatible with $\omega$.

In four dimensions, the existence of a locally complex elliptic symplectic structure forces the degeneracy locus to be of a very specific form.
\begin{proposition}\label{prop:topologydiv}
Let $M^4$ be a four-dimensional compact locally complex elliptic symplectic manifold with respect to a co-orientable elliptic divisor $I_{|D|}$. Then the connected components of $D$ are tori, spheres which intersect themselves in one point, or necklaces of spheres.
\end{proposition}
\begin{proof}
It follows from the normal form from \cite{CG17}, that over $D(1)$ the modular vector field of the underlying Poisson structure is nowhere-vanishing. In particular, if a connected component $D'$ of $D$ is smooth and co-orientable, then it is orientable and has a nowhere vanishing vector field, hence is diffeomorphic to a torus.

Assume that a connected component $D'$ is immersed but not embedded and let $\tilde{D}'\rightarrow D'$ denote the immersion. Because $D' \backslash D(2)$ is a Poisson submanifold, the modular vector field of the Poisson structure is tangent to $D' \backslash D(2)$ and is nowhere-vanishing on $D' \backslash D(2)$. Moreover, in local coordinates as in Proposition~\ref{prop:normformelliptic} the modular vector field takes the form $-r_1\partial_{r_1}-r_2 \partial_{r_2}$, which lifts to a vector field on $\tilde D'$ with positive zeros at the pre-image of $D(2)$. Therefore each component of $D'$ is an orientable surface with positive Euler characteristic (equal to the number of points in $\tilde D'$ that map to $D(2)$. That is, each component of $D'$ is a sphere with two points in $D(2)$. Thus either the component of $D'$ intersects itself in one point, or it intersects another component(s) at two points. That component can in turn intersect the previous component or some other component. By compactness there are only finitely many components and thus these spheres have to form a necklace.
\end{proof}
Now we address the question of when a locally complex elliptic symplectic structure is actually induced by a complex log symplectic form. To do so we introduce the following notion.

\begin{definition}\label{def:checkmarktimes}
Let $M^4$ be an oriented manifold with a co-oriented elliptic divisor $(I_{\abs{D}},\frak{o})$. For each  $p\in D(2)$ let $D_1,D_2$ be the corresponding local normal crossing divisors.
\begin{itemize}
\item We say that the {\bf  intersection index of $p$ is $1$} if the isomorphism $T_p M \simeq N_p D_1 \oplus N_p D_2$ is orientation-preserving;
\item We say that the {\bf  intersection index of $p$ is $-1$} otherwise.\qedhere
\end{itemize}
\end{definition}

Since $ND_1$ and $ND_2$ are both two-dimensional their choice of ordering does not effect the orientation of the resulting direct sum $ND_1 \oplus ND_2$.

An elliptic symplectic structure always provides a preferred orientation for $M$. If the structure is also locally complex, we can verify whether an intersection point in $D(2)$ is positive or not by considering the values of the residues at that point.
\begin{lemma}
Let $(M^4,I_{\abs{D}},\frak{o})$ be a manifold endowed with a co-oriented elliptic divisor, and let $\omega$ be a locally complex elliptic symplectic form. 
Then $p \in D(2)$ is positive with respect to the orientation induced by $\omega$ if and only if  \eqref{eq:loc complex 1} holds.
\end{lemma}

This can be rephrased in more global terms.

\begin{lemma}\label{lem:checkmarkloccmplx}
Let $M^4$  be a manifold endowed with a co-oriented elliptic divisor $(I_{\abs{D}},\frak{o})$  and an elliptic symplectic form, $\omega$. The triple $(I_{|D|},\frak{o},\omega)$ is in the image of the map in Theorem \ref{th:correspondence2} if and only if $\omega$ is locally complex and each point in $D(2)$ has positive index.
\end{lemma}
\begin{proof}
By Lemma \ref{lem:degree2impart},  $\omega \in \Omega^{2}_{\Im}(M)$ if and only if \eqref{eq:loc complex 1} holds, which by the previous Lemma corresponds to all intersection points being positive.
\end{proof}
Whenever a locally complex elliptic symplectic structure has a divisor with a few negative intersection points, we may try to fix this by flipping the co-orientation of one of the components arriving at that point. The problem with this is that such a change of co-orientation will also change the sign at `the other' intersection point of that component. Reflecting on this for a moment we see that the parity of the number of negative points is the relevant piece of data.
\begin{definition}
Let $(M,I_{|D|},\frak{o})$ be a four-manifold with co-oriented elliptic divisor and let  $\omega$ be a locally complex elliptic symplectic structure on $M$. If we denote by $\epsilon_p$ the index of $p \in D(2)$, the {\bf parity} of $\omega$ on a connected component $D'$ of $D$ is given by
\[
\varepsilon_{\omega,D'} = \Pi_{p\in D'(2)}\epsilon_p.
\]
If $D$ is connected, we omit $D$ from the notation and refer to $\varepsilon_\omega$ as the {\bf parity} of $\omega$.
\end{definition}

\begin{lemma}\label{lem:paritycomplexlog}
Let $(M,I_{|D|},\frak{o})$ be a four-manifold with co-oriented elliptic divisor and let $\omega \in \Omega^2(\elli)$ be a locally complex elliptic symplectic form. Then:
\begin{itemize}
\item  For each connected component $D'$ of $D$, the parity, $\varepsilon_{\omega, D'}$, does not depend on the choice of co-orientation;
\item  We have $\varepsilon_{\omega, D'} = 1$ for all connected components $D'$ of $D$ if and only if  there is a co-orientation $\frak{o}'$ of $D$ for which $(I_{|D|},\frak{o}',\omega)$ is in the image of the map in Theorem \ref{th:correspondence2}.\qedhere
\end{itemize}

\end{lemma}
\begin{proof}
The proof relies on Proposition \ref{prop:topologydiv}, which describes what a connected component of $D$ looks like. The case when $D$ is smooth and co-orientable was treated in \cite{CG17} so we need to consider the cases when $D$ has self-intersections.

If $D'$ has only one component which intersects itself in one point, $p$ if we change the co-orientation of $D'$ we change the co-orientation of both strands arriving at $p$ and hence the index of $p$ does not change. So $\varepsilon_{\omega,D'} =1$ if and only if the index of $p$ is 1 which, by Lemma \ref{lem:checkmarkloccmplx}, happens if and only if $(I_{|D|},\frak{o},\omega)$ is in the image of the map in Theorem~\ref{th:correspondence2} in a neighbourhood of $D'$.

In general, since each component has two intersection points, changing its co-orientation changes two signs and hence the parity, $\varepsilon_{\omega,D'}$, remains unchanged by this operation. The proof that this is the only relevant invariant can be done intuitively by breaking the necklace of spheres, $D'$, at one intersection point, so that we get an array of spheres and then fixing the co-orientation of `the first' component, $D_1$, of this array. Then we inductively keep or change the co-orientations of the next components one by one depending on whether the index of the next intersection point is positive or negative until we get to the last sphere, $D_n$. Since $D'$ is a necklace, not an array, there is one last index to be computed, namely the one between $D_n$ and $D_1$. Since the parity of the number of negative indexed points is fixed, this last index is positive if the parity is positive. Hence, again by Lemma \ref{lem:checkmarkloccmplx}, the parity is positive if and only if $(I_{|D|},\frak{o},\omega)$ is in the image of the map in Theorem \ref{th:correspondence2} in a neighbourhood of $D'$.

It is clear that $(I_{|D|},\frak{o},\omega)$ is in the image of the map in Theorem \ref{th:correspondence2} if and only if for each component $D'$ of $D$, $(I_{|D|},\frak{o},\omega)$ is in the image of the map in Theorem~\ref{th:correspondence2} in a neighbourhood of $D'$.
\end{proof}
\subsection{Smoothening self-crossing stable structures}
In this section we will show that if a four-dimensional manifold admits a stable generalized complex structure, then it also admits a smooth stable generalized complex structure.

Given $\varepsilon > 0$ consider the following two stable generalized complex structures on $\cc^2$:
\begin{equation}\label{eq:structures}
\rho_0 = \lambda z_1z_2 + dz_1\wedge dz_2, \quad \rho_1 = (\lambda z_1z_2 + \varepsilon) + dz_1 \wedge dz_2,
\end{equation}
which determine,  respectively,  the  complex log divisors:
\[
I_{D_0} = \inp{z_1z_2},\quad I_{D_1} = \inp{\lambda z_1z_2+\varepsilon},
\]
and corresponding complex log symplectic forms
\begin{equation}\label{eq:structures2}
\sigma_0 =  \frac{ dz_1 \wedge dz_2}{\lambda z_1z_2} \qquad \mbox{ and } \qquad  \sigma_1 = \frac{ dz_1 \wedge dz_2}{\lambda z_1z_2 +\varepsilon}.
\end{equation}

\begin{lemma}\label{lem:surglemma}
Let $\sigma_0$ and $\sigma_1$ be the complex log symplectic forms in \eqref{eq:structures2}. Then there are annuli $A_0,A_1  \subset \cc^2$ and a diffeomorphism $\Phi\colon A_0 \rightarrow A_1\subset \cc^2$ which is a morphism of divisors between $I_{D_1}$ and $I_{D_0}$ and satisfies $\Im^* \sigma_0 = \Phi^* \Im^* \sigma_1$. Moreover, the map $\Phi$ is ambient isotopic  to the natural inclusion $\iota\colon A_0 \to \cc^2\backslash \{0\}$.
\end{lemma}
\begin{proof}
The proof relies on a version of the Moser argument: We will find annuli  $A_0$ and $A_1$ together with a diffeomorphism, $\varphi\colon A_0 \to A_1$, with the following properties
\begin{itemize}
\item $\varphi$ is a morphism of divisors between $I_{D_1}$ and $I_{D_0}$,
\item $\varphi^*\Im^*\sigma_1$ lies in the same cohomology class of $\Im^*\sigma_0$ and
\item  the line connecting $\varphi^*\Im^*\sigma_1$  and $\Im^*\sigma_0$ is made of symplectic forms.
\end{itemize}
Once we have found such a $\varphi$, it is clear that the result follows from the Moser argument.

We start by considering $A \subset \cc^2$, the complement of the polydisc of radius $2\varepsilon/|\lambda|$ on $\cc^2$, that is
\[A = \{(z_1,z_2)\colon |z_1| > 2\varepsilon/|\lambda| \mbox{ or } |z_2| > 2\varepsilon/|\lambda|\}.
\]
On $A$, we let
\begin{align*}
U &= \{(z_1,z_2) \in A\colon |z_1| < \sqrt{\varepsilon/|\lambda|} \mbox{ or }|z_2|< \sqrt{\varepsilon/|\lambda|}\},\\
V &= \{(z_1,z_2) \in A\colon |z_1| < 2\sqrt{\varepsilon/|\lambda|} \mbox{ or }|z_2|< 2\sqrt{\varepsilon/|\lambda|}\}.
\end{align*}
Notice that both $U$ and $V$  consist of two components, namely $U_1, V_1$, which are neighbourhoods of $\set{z_2 = 0} \cap A$, and $U_2, V_2$, which are neighbourhoods of $\set{z_1 = 0} \cap A$. Also, notice that the zeros of $\lambda z_1 z_2 + \varepsilon$ in $A$ also lie inside $U$, that is $U$ is also a neighbourhood of the zero locus of the divisor associated to $\rho_1$ (See Figure \ref{fig:picture}).

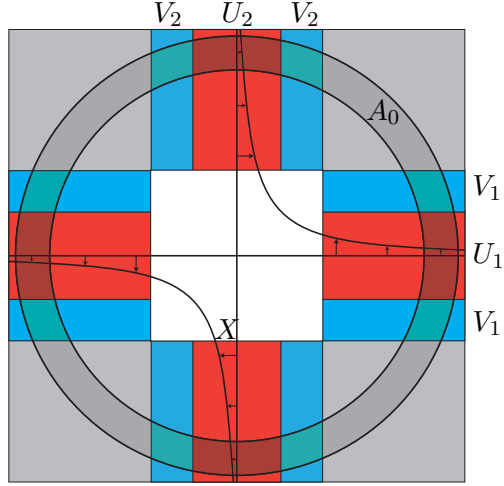
\begin{figure}
\begin{overpic}[unit=1mm,scale=0.2]{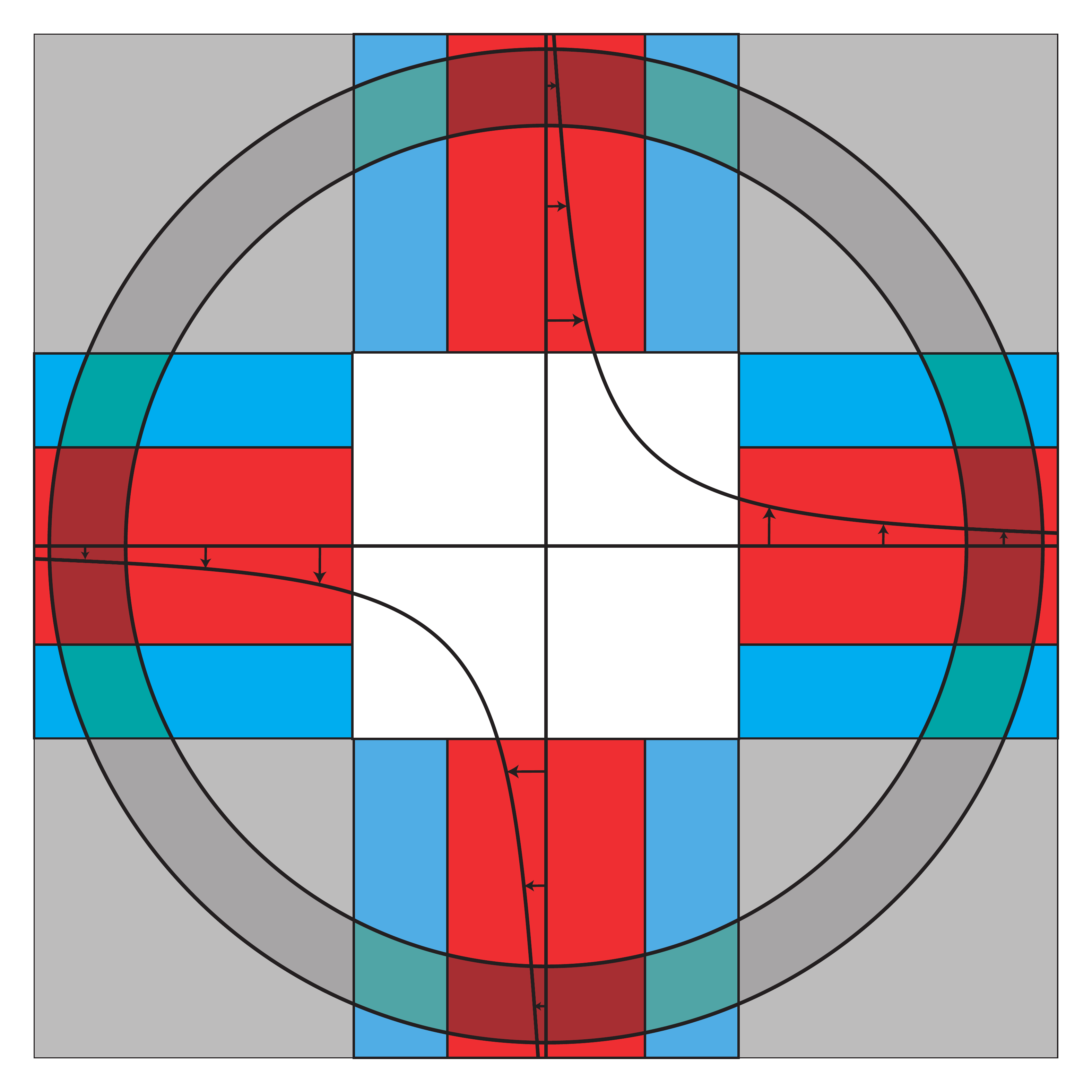}
\put(63,31){$U_1$}
\put(63,40){$V_1$}
\put(63,22){$V_1$}
\put(30,63){$U_2$}
\put(21,63){$V_2$}
\put(39,63){$V_2$}
\put(29,21){$X$}
\put(49,50){$A_0$}
\end{overpic}
\caption{The  divisor associated to $\rho_0$ are the coordinate axes while the divisor associated to $\rho_1$ is the hyperbola. In the complement of the polydisc of radius $2\varepsilon/|\lambda|$ both of these lie in $U$. The vector field $X$ flows the coordinate axes to the intersection of the hyperbola with $A$ and vanishes in the grey area.}\label{fig:picture}
\end{figure}

Let $\psi\colon [0,\infty) \to [0,1]$ be a monotone bump function which is $1$ on $[0,\tfrac{3}{2}\sqrt{\varepsilon/|\lambda|}]$ and zero on $[2\sqrt{\varepsilon/|\lambda|},\infty)$, and consider the following complex vector field:
\begin{align*}
X:=
\begin{cases}
-\frac{\psi(|z_2|)\varepsilon}{\lambda z_1}\partial_{z_2} \quad &\text{on } V_1,\\
-\frac{\psi(|z_1|)\varepsilon}{\lambda z_2}\partial_{z_1} &\text{on } V_2,\\
0 &\text{otherwise}.
\end{cases}
\end{align*}
The time-one flow of the real part of this vector field defines a diffeomorphism $\varphi\colon A \rightarrow  A$.

Since in $U_1$, $X = -\frac{\varepsilon}{\lambda z_1}\partial_{z_2}$, the flow of its real part is the shear transformation
\begin{equation}\label{eq:shear}
\varphi_t(z_1,z_2) = (z_1, z_2 -t\frac{\varepsilon}{\lambda z_1})
\end{equation}
as long as the flow remains in $U_1$. Hence the time-one flow satisfies
\begin{equation}\label{eq:U_1}
\varphi^*(\lambda z_1 z_2 +\varepsilon) = \lambda z_1 (z_2 -\frac{\varepsilon}{\lambda z_1}) + \varepsilon = \lambda z_1 z_2. 
\end{equation}
A similar computation holds in $U_2$ and hence $\varphi^*I_{D_1} = I_{D_0}$.
 
Since $\Im^*$ is a cochain map, to prove that $[\Im^*\varphi^*\sigma_1] = [\Im^*\sigma_0]$ it is enough to prove that  $[\varphi^*\sigma_1] = [\sigma_0] \in H^2(\mathcal{A}_{D_0})$.  By Theorem~\ref{thm:complexcom} we have that $H^2(A,\mathcal{A}_{D_0}) \simeq H^2(A\backslash D_0)$. Because $A \backslash D_0$ deformation retracts onto a torus,  for example, $T^2 = \set{\abs{z_1} = R, \abs{z_2} = R}$ with $R> 2\sqrt{\varepsilon/|\lambda|}$, we see that $H^2(A \backslash D_0) \simeq \rr$ and we are left to show that $\int_{T^2}\sigma_0 = \int_{T^2}\varphi^*\sigma_1$. This torus is disjoint from $V$, hence  $\varphi$ is the identity near $T^2$ and we compute
\begin{align*}
\int_{T^2} \varphi^*\sigma_1 &= \int_{\abs{z_1}=R} \int_{\abs{z_2}=R} \frac{dz_1 \wedge dz_2}{\lambda z_1z_2 + \varepsilon} \\
&=\int_{\abs{z_1}=R}\int_{\abs{z_2}=R} \frac{1}{\lambda z_1} \frac{dz_1 \wedge dz_2}{z_2 + \frac{\varepsilon}{\lambda z_1}}\\
&= 2\pi i \int_{\abs{z_1} = R} \frac{1}{\lambda} \frac{dz_1}{z_1} = \frac{-4\pi}{\lambda},
\end{align*}
which coincides with the integral of $\sigma_0$ over the same torus. Therefore we conclude that indeed $[\sigma_0] = [\varphi^*\sigma_1] \in H^2(A,\mathcal{A}_{D_0})$, and consequently that $[\Im^* \sigma_0] = [\varphi^* \Im^* \sigma_1] \in H^2(A,\mathcal{A}_{\abs{D_0}})$.

Next we will find a radius $r$ such that on $A\backslash B_r$ (the complement of the ball of radius $r$) the form $\sigma_t := t\varphi^*\sigma_1 + (1-t)\sigma_0$ is complex log symplectic for all values of $t$. To find $r$, we will study separately the behaviour of $\sigma_t$ in three regions: $U$, $A\backslash V$ and $V \backslash U$.

Let us start with the region $U$.  If $r>0$ is  large enough, the vector field $X$ becomes small on $A\backslash B_r$ and its time-one flow starting at $U_1$ is the shear transformation \eqref{eq:shear}, hence  we have $\varphi^*\sigma_1 = \sigma_0$. The same argument holds at $U_2$ and hence, for $r$ large enough, on $U \backslash B_r$,  $\sigma_t = \sigma_0$ is symplectic for all $t$.

Next we consider $A\backslash V$. In this region, $\varphi = \text{id}$, and therefore
\begin{align*}
\sigma_t &= t \frac{dz_1 \wedge dz_2}{\lambda z_1 z_2 + \varepsilon} + (1-t) \frac{dz_1 \wedge dz_2}{\lambda z_1 z_2}\\
&= \frac{\lambda z_1 z_2 + (1-t)\varepsilon}{(\lambda z_1z_2 + \varepsilon)(\lambda z_1 z_2)}dz_1 \wedge dz_2.
\end{align*}
Since  in this region $|z_1| > 2\sqrt{\varepsilon/|\lambda|}$ and $|z_2|>2\sqrt{\varepsilon/|\lambda|} $, the form above does not vanish and  hence its imaginary part is symplectic. 

Finally, we deal with $V\backslash (U\cup B_r)$. We will focus on $V_1\backslash U_1$. Apply the Mean Value Theorem to the map $t \mapsto \Im^*(\varphi_X^t)^*\sigma_1(p)$, we obtain the estimate
\[|\Im^*\sigma_1(p) -  \Im^*\varphi^*\sigma_1(p)| < |\Im^*\mathcal{L}_X\sigma_1(p')|.
\]
for $p'$ in line segment between $p$ and $\varphi(p)$. Using the explicit forms of $X$ and $\sigma_1$ we find
 \[|\Im^*\sigma_1 -  \Im^*\varphi^*\sigma_1| <  O(1/r^2).
\]
We can also estimate the difference
\[
|\Im^*\sigma_0 - \Im^*\sigma_1 | = \left|\Im^*\left(\frac{\varepsilon dz_0 dz_1}{\lambda z_0 z_1(\lambda z_0 z_1 + \varepsilon)}\right)\right| =O(1/r^2).
\]
Therefore, on $V_1\backslash (U_1 \cup B_r)$  we can estimate
\[
\Im^*\sigma_t =  \Im^*\sigma_0 + t(\Im^*\varphi^*\sigma_1-\Im^*\sigma_0) = \Im^*\sigma_0 - t((\Im^*\varphi^*\sigma_1-\Im^*\sigma_1)  + (\Im^*\sigma_1-\Im^*\sigma_0)) = \Im^*\sigma_0    - t O(1/r^2).
\]
Since $\Im^*\sigma_0 = O(1/r)$ the above is symplectic as long as $r$ is large enough. 

Therefore, by picking $r$ as above and $R >r$ we can apply the Moser Lemma to the annulus $A_0 = B_R\backslash B_r$ to find a diffeomorphism $\tilde\Phi\colon A_0 \to \tilde\Phi(A_0)$ such that $\tilde\Phi^* (\varphi^*\Im^*\sigma_1) = \Im^*\sigma_0$.   That is, $\Phi = \varphi \circ \tilde \Phi \colon A_0 \to (\varphi \circ \tilde \Phi)(A_0)$ is the diffeomorphism we were looking for. Since $\Phi$ is the composition of flows of vector fields it is ambient isotopic to the inclusion $A_0 \to \cc^2 \backslash \{0\}$.
\end{proof}
Using this lemma we can now prove that every locally complex elliptic symplectic structure can be changed into a smooth one by performing a surgery.
\begin{theorem}\label{th:smoothing}
Let $(M^4,I_{\abs{D}})$ be a manifold endowed with a co-oriented elliptic divisor, and let $\omega \in \Omega^2(\elli)$ be elliptic symplectic. Then:
\begin{itemize}
\item If $\omega$ is locally complex, then it can be changed into a smooth elliptic symplectic form with zero elliptic residue $\tilde\omega$;
\item The resulting structure will be induced by a stable generalized complex structure if and only if the original structure was.
\end{itemize}
\end{theorem}
\begin{proof}
Since $D(2)$ has codimension four and is compact, it is a finite collection of points. Because $\omega$ is locally complex, given $p \in D(2)$ there is a neighbourhood of $p$ in which $\omega$ is the imaginary part of some complex log symplectic form $\sigma$. By Lemma \ref{thm:stablelocform} there exists complex coordinates $(z_1,z_2)$ on a ball, $B$,  around $p$ on which $\sigma$ is ($B$-field equivalent) to $\lambda d \log z_1 \wedge d \log z_2$. Note that as $\sigma$ is scaling-invariant in both the $z_1$- and $z_2$-direction, we can ensure that this ball is as large as necessary. Therefore we can apply Lemma \ref{lem:surglemma} to find an annulus $A_0 \subset B$ together with its accompanying diffeomorphism $\Phi\colon A_0 \rightarrow A_1$. We denote by $B' \subset B$ the inner ball enclosed by $A_0$. Let $\tilde{B} \subset \cc^2$ be the ball enclosed by the outer boundary of the annulus $A_1$. Because $\Phi$ is ambient isotopic to the inclusion of $A_0$ in $D^4\backslash \{0\}$, we have $M \simeq M\backslash B' \cup_{A,\phi} \tilde{B}$. If we endow $\tilde{B}$ with the smooth stable structure $\rho_1$ in Equation~\eqref{eq:structures}, we see by Lemma~\ref{lem:surglemma} that the map $\Phi$ is an elliptic symplectomorphism. We conclude that $M$ obtains an elliptic symplectic structure with $D(2)$ consisting of one point less. Moreover, as the surgery is purely local in nature we can ensure that the structure does not change around the remaining points in $D(2)$. By performing this procedure for all points in $D(2)$ we conclude that $M$ admits a smooth elliptic symplectic structure with zero elliptic residue. 

For the second part of the theorem, if $\omega$ was induced by a stable structure, then the local coordinates from Theorem \ref{thm:stablelocform} would be orientation-preserving. As all maps used in the surgery are orientation-reserving we conclude that the resulting divisor is co-orientable. We conclude that $M$ admits a smooth stable generalized complex structure. If $\omega$ is not induced by a stable structure, then there is at least one set of coordinates obtained from Theorem \ref{thm:stablelocform} which is orientation-reversing. Therefore the resulting divisor will not be co-orientable, which finishes the proof.
\end{proof}

\section{Connected sums}\label{sec:connectedsum}
In this section we will introduce a connected sum operation for elliptic symplectic structures in four dimensions. To do so we will make use of normal form results obtained in Section~\ref{sec:4d}. The operation will be phrased in terms of locally complex  symplectic forms (Definition~\ref{def:locallycomplex}) and it will be useful for us to keep track of the index of points as we perform the connected sum.
\subsection{Glueing divisors}
We will perform connected sums on elliptic symplectic four-manifolds at points which lie in their respective sets $D(2)$. In arbitrary dimensions it is possible to take connected sums of elliptic divisors at points with the same intersection number.
\begin{lemma}\label{lem:gluediv}
Let $(M^n,I_{\abs{D_M}})$, $(N^n,I_{\abs{D_N}})$ be two oriented manifolds endowed with elliptic divisors, and let $p \in D_M(k)$ and $q \in D_N(k)$ for $k \in \nn$. Then $M \#_{p,q} N$ admits an elliptic divisor $I_{\abs{D}}$, for which the inclusions $M\backslash \{p\}, N\backslash\{q\} \rightarrow (M \#_{p,q} N)$ are morphisms of divisors.
\end{lemma}
\begin{proof}
Using Lemma \ref{lem:locMB} there exist coordinates around $p$ and $q$ such that both divisors are precisely given by the ideal $I = \inp{r_1^2\cdot \ldots \cdot r_k^2}$. To take the connected sum of $M$ and $N$ at $p$ and $q$, we need to use an orientation-preserving diffeomorphism, $F$, from an annulus around $p$ to an annulus around $q$, which reverses the co-orientation of the sphere and for which $F^*I = I$. Here we consider oriented charts defined in neighbourhoods of $p$ and $q$ which map to the unit ball in $\rr^n$ and we will use the diffeomorphism given in spherical coordinates by
\begin{equation}\label{eq:inversion}
F : \rr^n\backslash \set{0} \rightarrow \rr^n \backslash \set{0},\quad (r,\phi_1,\ldots,\phi_n) \mapsto (r^{-1},\phi_1,\ldots,-\phi_n).
\end{equation}
To verify that $F^*I = I$ we write $r_i^2 = r^2 \psi_i(\phi_1,\ldots,\phi_n)$, where $\psi_i$ is a function which only depends on the angular coordinates and satisfies $\psi_i(\phi_1,\dots, -\phi_n) = \psi_i(\phi_1,\dots, \phi_n)$. We have
\begin{equation*}
	F^*(r_i^2) = \frac{1}{r^2}\psi_i(\phi_1,\ldots,\phi_n), \text{ and thus } r^4F^*(r_i^2) = r_i^2.
\end{equation*}
Hence $r^{4k}F^*(r_1^2\cdot \ldots \cdot r_k^2) = r_1^2\cdot \ldots \cdot r_k^2$, and as $r^{2k}$ is a non-zero function on $\rr^n\backslash \set{0}$ we conclude that $F^*I=I$.
\end{proof}
There are a couple of points about this construction that we should stress.
\begin{remark}\label{rem:multipledivs}
There is some freedom in the glueing of elliptic divisors. Given a choice of local coordinates $(z_1,\ldots,z_m)$ around $p$ and $q$ we can furthermore compose the map $F$ by a permutation of the first $k$ coordinates. Note that this does not change $M \#_{p,q} N$, but it could change the topology of the zero locus of the divisor. Because of this ambiguity in ordering, there are potentially $k!$ different topological types for the vanishing locus of the divisor on the connected sum $M \#_{p,q} N$.
\end{remark}

\begin{remark}\label{rem:orientation?}
The reason for flipping the sign of $\phi_n$ in the map in \eqref{eq:inversion} is that we assumed that the coordinates we chose are compatible with the orientations of $M$ and $N$. If for some reason only one of the chosen coordinates was compatible with the orientations of $M$ and $N$, we should not flip the sign of $\phi_n$. 
\end{remark}

\subsection{Glueing symplectic structures}
We will introduce a connected sum operation for elliptic symplectic structures in dimension four. The existence of such an operation contrasts starkly with ordinary symplectic geometry, where connected sums are not possible above dimension two. We first note that $F$ introduced in \eqref{eq:inversion} is a local symplectomorphism for a specific order of the coordinates.
\begin{lemma}\label{lem:glue}
Let $(z_1,z_2)$ be complex coordinates on $M = \cc^2\backslash \set{0}$, and consider the elliptic symplectic structure on $M$ given by
\begin{align*}
\omega = \Im^*(i\, d\log z_1 \wedge d\log z_2).
\end{align*}
Using polar coordinates
\[(z_1,z_2) = (r \cos \phi_1,r\sin \phi_1 \cos \phi_2,r\sin \phi_1\sin \phi_2 \cos \phi_3,r\sin \phi_1\sin \phi_2 \sin \phi_3),
\]
the map defined in \eqref{eq:inversion} satisfies
\[ F^*\omega = -\omega.\]
\end{lemma}
\begin{proof}
The proof is a direct computation using that we can express $F$ in complex coordinates as
\[
F(z_1,z_2) = \frac{1}{r^2}(z_1,\bar{z}_2).\qedhere
\]
\end{proof}
\begin{definition}
Let  $\omega \in \Omega^2(\elli)$ be a locally complex elliptic symplectic form. We say that $\omega$ has \textbf{imaginary parameter} at a point $p \in D(2)$ if 
\[
\Res_{r_1\theta_2}\omega(p) = \Res_{r_2\theta_1}\omega(p) = 0.\qedhere
\]
\end{definition}
Note that this definition does not depend on a choice of co-orientation.
\begin{remark}\label{rem:imaginary parameter}
By Proposition~\ref{prop:normformelliptic}, an elliptic symplectic form with imaginary parameter $\omega$ is locally isomorphic to
\begin{equation}\label{eq:imaginary parameter}
	\lambda\Im^*(id\log z_1 \wedge d\log z_2),
\end{equation}
for  an appropriate choice of complex coordinates, where $\lambda = \Res_{r_1 r_2}(\omega)$. This justifies the terminology. It is also immediate that these complex coordinates are compatible with the orientation defined by the complex structure. 

Notice however that if the divisor is co-oriented, the complex coordinates above may not be compatible with the co-orientations. If we require compatibility between co-orientation and complex coordinates, $\omega$ will be isomorphic to either the form above or
\begin{equation*}
	\lambda\Im^*(id\log z_1 \wedge d\log \bar z_2).
\end{equation*}
In the latter case, the complex coordinates and symplectic structure induce opposite orientations.
\end{remark}
We can now turn to the main result of this section, namely the connected sum operation in dimension four for elliptic symplectic structures with imaginary parameter at points in $D(2)$ whose imaginary parameters match in absolute value.
\begin{theorem}\label{thm:ellipticglueing}
Let $(M,I_{\abs{D_M}})$ and $(N,I_{\abs{D_N}})$ be four-manifolds with elliptic divisors. Let $\omega_1,\omega_2$ be locally complex elliptic symplectic forms with imaginary parameters at $p \in D_M(2)$ and $q\in D_N(2)$, and denote by $D'_M$ and $D'_N$ the connected components of the divisor containing $p$ and $q$. If $\abs{\Res_{r_1r_2}\omega_1(p)} = \abs{\Res_{r_1r_2}\omega_2(q)}$, then:
\begin{itemize}
\item The connected sum $M \#_{p,q} N$ admits a locally complex elliptic symplectic form, $\omega$, for which the inclusions  $(M \backslash \{p\},\omega_1), (N \backslash \{q\},\omega_2)  \hookrightarrow (M \#_{p,q} N,\omega)$ are elliptic symplectic maps.
\item If either $D'_M(2)$ or $D'_N(2)$ has more than one point, the parity of $D'_{M\# N}$ is given by
\[ \varepsilon_{D'_{M\# N}} = - \varepsilon_{D'_M} \varepsilon_{D'_N}.\]
\item If $D'_M(2) =\{p\}$ and $D_N(2) = \{q\}$, then $D'_{M\#N}$ is co-orientable if and only if $p$ and $q$ have opposite parities.
\end{itemize}
\end{theorem}
\begin{proof}  We will prove the claims of the theorem in turn. The tools needed are the normal form for locally complex elliptic symplectic structures with imaginary parameter from Remark~\ref{rem:imaginary parameter} and the local symplectomorphisms from Lemma~\ref{lem:glue}.

To prove the first claim we choose complex coordinates in a neighbourhood of $p$ which render the symplectic structure in the form \eqref{eq:imaginary parameter} and do the same for $q$, but reverse their order so that $\Res_{r_1r_2}\omega_1(p) = -\Res_{r_1r_2}\omega_2(q)$. 

For this choice of coordinates, if we use $F$ to perform the connected sum, Lemma \ref{lem:glue} implies that the structures on $M\backslash\{p\}$ and $N\backslash\{q\}$ agree on their overlap on  $M\#_{p,q} N$ which therefore inherits an elliptic symplectic structure.

Now we move to the second claim. Choose co-orientations for $D'_M$ and $D'_N$. If  $p$ and $q$ have the same index and, say, $D'_M(2)$ has more than one point, we can change the choice of co-orientation of one of the components arriving at $p$ which causes the index of $p$ to change sign. So we may assume without loss of generality that the indices of $p$ and $q$ are opposite. We assume that $p$ has negative index and $q$ has positive index as the other case is analogous.

In this case, the complex coordinates used in a neighbourhood of $p$ in the first claim only give the correct co-orientation of one of the components of $D'_M$ passing through $p$, say, the one given by $[z_2=0]$. That is the co-orientation of $[z_1=0]$ is determined by $d\bar{z_2}$ and the co-orientation of $[z_2=0]$ by $dz_1$. Since $q$ has positive index, we may assume that the complex coordinates chosen for $q$ are compatible with co-orientations. Finally we observe that the map $F$ from Lemma~\ref{lem:glue} sends the line $dz_1$ over $[z_2=0]$ to itself and sends the line $dz_2$ over $[z_1=0]$ to $d\bar{z_2}$. Therefore, the co-orientations of $D'_M$ and $D'_N$ are mapped to each other under $F$ and hence $D'_{M\# N}$ inherits a natural co-orientation from those of $D'_M$ and $D'_N$ and we can compute its index:
\[
\varepsilon_{D'_{M\#N}} = \Pi_{r\in D'_M(2), r\neq p} \epsilon_r  \Pi_{s\in D'_N(2), s\neq q} \epsilon_s = - \epsilon_p\epsilon_q \Pi_{r\in D'_M(2), r\neq p} \epsilon_r  \Pi_{s\in D'_N(2), s\neq q} \epsilon_s =  -\varepsilon_{D'_M}\varepsilon_{D'_N}.
\]
Finally we prove the last claim. If the indices of $p$ and $q$ are opposite, the previous argument shows that a choice of co-orientation for $D'_M(2)$ and $D'_N(2)$ induces a co-orientation for $D'_{M\# N}$, which is therefore co-orientable. If the indices of $p$ and $q$ agree, then the argument above shows that map $F$ matches co-orientations of one of the strands arriving at $p$ and $q$ and reverses the other. Since $D'_M$ and $D'_N$ are both connected, this implies that $D'_{M\# N}$ is not co-orientable.
\end{proof} 
Instead of performing connected sums of two manifolds we can also perform a connected sum of a manifold with itself (self-connected sum), that is we can glue neighbourhoods of points $p$ and $q \in M$ by an inversion on the annulus. In this case, if $M$ is connected, this operation corresponds to attaching a $1$-handle and hence the diffeomorphism type of the resulting space is always $M \# (S^1 \times S^3)$. In this context, Theorem~\ref{thm:ellipticglueing} becomes:
\begin{corollary}\label{prop:handleattachments}
Let $(M^4,I_{\abs{D}})$ be a four-manifold with an elliptic divisor and let $\omega$ be a locally complex elliptic symplectic form with imaginary parameters at $\{p,q\} \in D(2)$ with $p\neq q$. Denote by $D'_p$ and $D'_q$ the connected components of the divisor containing $p$ and $q$, respectively. If $\abs{\Res_{r_1r_2}\omega(p)} = \abs{\Res_{r_1r_2}\omega(q)}$, then:
\begin{itemize}
\item $M^4 \# (S^1\times S^3)$ admits a locally complex elliptic symplectic structure for which the inclusion $M^4 \backslash \{p,q\} \hookrightarrow (M^4 \# (S^1\times S^3))$ is an elliptic symplectic map.
\item If  $\{p,q\} \subsetneq D'_p(2) \cup D'_q(2)$  and $D'_p \neq D'_q$, the parity of corresponding divisor, $D'$,  in $M^4 \# (S^1\times S^3)$ is given by
\[ \varepsilon_{D'} = - \varepsilon_{D'_p} \varepsilon_{D'_q}.\]
\item If  $\{p,q\} \subsetneq D'_p(2) \cup D'_q(2)$  and $D'_p = D'_q$, the parity of corresponding divisor, $D'$  in $M^4 \# (S^1\times S^3)$ is given by
\[ \varepsilon_{D'} = - \varepsilon_{D'_p}.\]
\item If $\{p,q\} =  D'_p(2) \cup D'_q(2)$ then $D'$ is co-orientable if and only if $p$ and $q$ have opposite parities.
\end{itemize}
\end{corollary}
\begin{remark}
It might be more desirable to arrive at the conclusion of this corollary by producing a stable generalized complex structure on $S^1\times S^3$ for which $D(2) \neq \emptyset$ and then using Theorem~\ref{thm:ellipticglueing} to perform the connected sum. Unfortunately, presently we do not know if $S^1 \times S^3$ has such a structure. 
\end{remark}

\section{Examples}\label{sec:examples}
In this section we will use the connected sum procedure of Theorem~\ref{thm:ellipticglueing} to create several examples of elliptic symplectic structures and stable generalized complex structures. In order to do so we will start by constructing several building blocks, which are then combined via connect sum.

\subsection{Simple examples}
As we showed in Example \ref{th:CP2n}, $\cc P^2$ admits a stable generalized complex structure whose divisor is given by three lines. The next few examples show that $S^2 \times S^2$ has such a structure as well, while $\bar{\cc P^2}$ and $S^4$ do not, but do have locally complex elliptic symplectic structures.

\begin{example}[$\overline{\cc P^2}$]\label{prop:barcp2}
In this example we construct an elliptic symplectic form with imaginary parameter on $\bar{\cc P^2}$ for the elliptic divisor induced by $(\mathcal{O}(3),z_0z_1z_2)$. Needless to say, the structure we construct is not the imaginary part of a complex log symplectic form, since $\overline{\cc P^2}$ does not admit generalized complex structures.

Let $\mathcal{O}(1) \rightarrow \cc P^2$ be the dual of the tautological line bundle and for $i=0,1,2$ let $z_i \in \Gamma(\mathcal{O}(1))$ be the section induced by the homogeneous polynomial $z_i$ on $\cc^3$. Consider the three smooth complex log divisors, $D_i = (\mathcal{O}(1),z_i)$, let $D = (\mathcal{O}(3),z_0z_1z_2)$ be their product and let $|D|$ be the corresponding elliptic divisor. Using the underlying affine coordinates,
\begin{equation*}
	u_1 = \frac{z_1}{z_0}, \qquad u_2 = \frac{z_2}{z_0}, \qquad v_0 = \frac{z_0}{z_1}, \qquad v_2 = \frac{z_2}{z_1}, \qquad w_0 = \frac{z_0}{z_2}, \qquad w_1 = \frac{z_1}{z_2},
\end{equation*}
we define the following global elliptic two-form 
\begin{align*}
\omega := 
\begin{cases}
\Im^*(id\log \bar{u_1} \wedge d\log u_2) \quad &\text{if } z_0 \neq 0,\\
\Im^*(-id\log \bar{v_0} \wedge d\log v_2) \quad &\text{if } z_1 \neq 0,\\
\Im^*(-id\log \bar{w_1} \wedge d\log w_0) \quad &\text{if } z_2 \neq 0.\
\end{cases}
\end{align*}
It is immediate from the expression above that $\omega$ is locally complex with imaginary parameter. Moreover we see that it induces the orientation opposite from the usual complex structure on $\cc P^2$, hence it is a locally complex elliptic symplectic structure with imaginary parameter for every point in $D(2)$  on $\overline{\cc P^2}$. 

Finally, if we consider the co-orientation of the elliptic divisors induced by the complex log divisor in Example \ref{ex:complexlogcpn} we have that the points $D_0 \cap D_2$ and $D_0 \cap D_1$ have positive index, while the point in $D_1 \cap D_2$ has negative index. By Lemma \ref{lem:paritycomplexlog} we conclude that $\omega$ cannot be the imaginary part of a complex log symplectic form. Further, we observe that if we were to choose the opposite co-orientation for $D_0$, all intersection indices would be $-1$.

We can provide a simple picture to illustrate this and all the other examples in this section. Recall that $\cc P^2$ admits a singular torus fibration whose fibers are the orbits of the standard torus action on $\cc P^2$. The quotient space $\cc P^2/T^2$ is a triangle and the elliptic symplectic structure constructed above is invariant under this action (with symplectic fibers). The zero locus of the divisor is the pre-image of the edges of the triangle and points in $D(2)$ are the pre-images of the vertices. With this in mind, we use a triangle to represent $\cc P^2$ (or $\overline{\cc P^2}$) and decorate each vertex of the triangle with the intersection index of the corresponding point in $D(2)$ (see Figure \ref{fig:cp2bar}).\qedhere

\begin{figure}[h!]
\begin{overpic}[unit=1mm,scale=0.5]{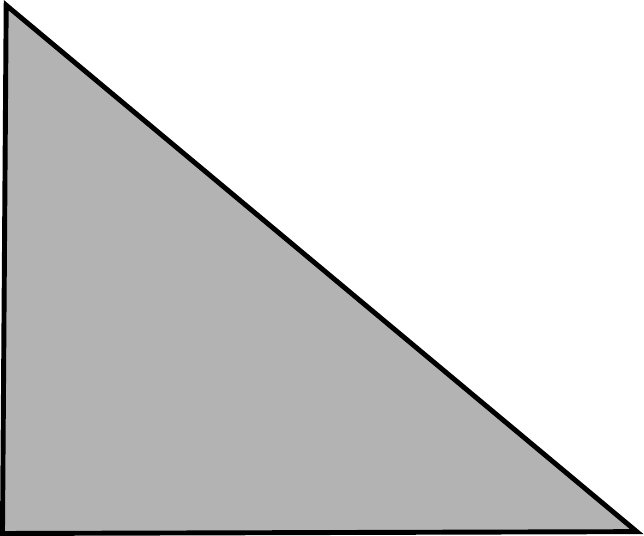}
\put(-6,-3){\color{ao(english)}$+1$}
\put(-6,28){\color{ao(english)}$+1$}
\put(33,-3){\color{red}$-1$}
\end{overpic}
\caption{We visualise $\bar{\mathbb{C}P^2}$ as its image under the moment map and each vertex in the triangle corresponds to a point in $D(2)$. We label the vertices with $\pm 1$ according to the intersection index of the corresponding point in $\overline{\cc P^2}$.}\label{fig:cp2bar}
\end{figure}
\end{example}

\begin{example}[$S^2 \times S^2$]\label{prop:S2xS2}
The manifold $S^2\times S^2$ admits a complex log symplectic structure $\sigma$ for which $D(2)$ consists of four points. The imaginary part of $\sigma$ is an elliptic symplectic form with imaginary parameter. Indeed, identifying $S^2$ with the extended complex plane, the vector field $z \partial_z$ vanishes transversely at $0$ and $\infty$ and hence in $S^2\times S^2$ (with complex coordinates $z$ and $w$), the bivector field $\pi = -izw \partial_z \partial_w$ is Poisson and determines a complex log divisor. Therefore we can use $\pi$ to deform the complex structure of $S^2 \times S^2$ into a stable generalized complex structure (as in Example \ref{exa:holpoissonsgcs}). A direct check shows that this structure has imaginary parameter at all points in $D(2)$. Since this stable generalized complex structure is obtained from a holomorphic Poisson structure, the natural co-orientation of each component of the divisor (induced by the complex structure) makes all intersection indices positive. Yet, by changing co-orientations, we can arrange that any pair or all four points in $D(2)$ have negative index.

Just as for $\overline{\cc P^2}$, we can provide an illustration for this structure using the toric description of $S^2 \times S^2$ (see Figure \ref{fig:S2xS2}).
\end{example}
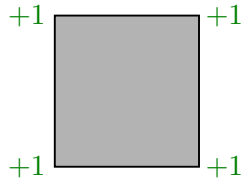
\begin{figure}[h!]
\begin{overpic}[unit=1mm,scale=0.5]{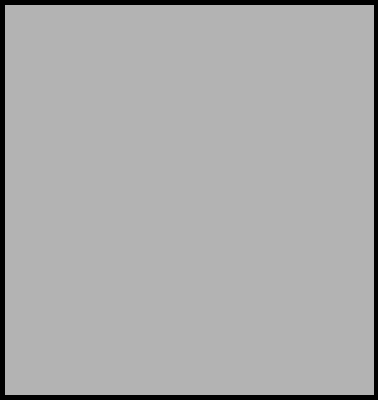}
\put(-6,-1){\color{ao(english)}$+1$}
\put(-6,19){\color{ao(english)}$+1$}
\put(20,-1){\color{ao(english)}$+1$}
\put(20,19){\color{ao(english)}$+1$}
\end{overpic}
\caption{We visualise $S^2\times S^2$ as the image under the map given by the two height functions. We label the vertices with $\pm 1$ according to the intersection index of the corresponding points in $S^2\times S^2$. Different choices of co-orientations yield different sign combinations at the vertices.}\label{fig:S2xS2}
\end{figure}
\begin{example}[$S^4$]\label{prop:S4}
The manifold $S^4$ admits an elliptic symplectic form with imaginary parameter with divisor consisting of two copies of $S^2$ intersecting each other at the north and south pole.

Consider two copies of $D^4$ and endow one copy with the two-form $\omega := \Im^*(id\log z_1 \wedge d\log z_2)$ and the other copy with $-\omega$. Using the map $F$, as in Lemma \ref{lem:glue}, we can glue an annulus in one of the disks to an annulus in the other disk while preserving the elliptic symplectic structures. The resulting manifold is diffeomorphic to $S^4$ and the divisors intersect at the points $(z_1,z_2) = (0,0)$ in both copies of $D^4$, which correspond to the north and south pole of the sphere. Because $F$ involves a complex conjugation, a choice of co-orientations for which one point in $D(2)$ has positive index causes the other point to have negative index. As in the previous examples, $S^4$ admits a natural torus action which rotates each complex coordinate in $D^4$ for which the hyperplanes $[z_i=0]$ have $S^1$ isotropy and the north and south poles are fixed points. This allows us to produce a two-dimensional   illustration of this structure (see Figure \ref{fig:S4}).
\end{example}
\begin{figure}[h!]
\begin{overpic}[unit=1mm,scale=0.5]{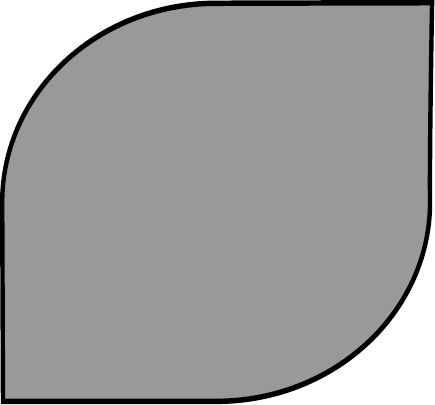}
\put(-6,-1){\color{ao(english)}$+1$}
\put(23,19){\color{red} $-1$}
\end{overpic}
\caption{We visualise $S^4$ as its quotient by the standard torus action: the edges correspond to the hyperplanes $[z_i=0]$ and the corners to the north and south poles.}\label{fig:S4}
\end{figure}
\begin{example}\label{ex:kleinsphere}
Because $S^4$ admits an elliptic symplectic form with imaginary parameter, by Example~\ref{prop:S4}, we can use Theorem \ref{th:smoothing} to obtain a smooth elliptic symplectic structure on $S^4$. However the degeneracy locus is non co-orientable, because the smoothing process did not preserve co-orientations. Because $S^4$ is orientable, we conclude that the degeneracy locus has to be non-orientable. Since the modular vector field is tangent to the degeneracy locus and nowhere zero, we see that the degeneracy locus is diffeomorphic to a Klein bottle.
\end{example}
\subsection{Main class of examples}
In this section we will combine all of our four-dimensional results to create new examples of elliptic symplectic and stable generalized complex structures on a large class of four-manifolds exhibited as connected sums.

\begin{theorem}\label{th:examples}
The manifolds in the following two families admit stable generalized complex structures:
\begin{enumerate}
\item $X_{n,\ell}: = \# n (S^2\times S^2)\# \ell (S^1\times S^3)$, with  $n,\ell \in \nn$;
\item $\hat X_{n,m,\ell}:=\# n \mathbb{C} P^2 \#m \overline{\mathbb{C} P^2}\# \ell(S^1 \times S^3)$, with $n,m\ell \in \nn$,
\end{enumerate}
as long as $1 - b_1 + b_2^+$ is even and the Euler characteristic is non-negative.
\end{theorem}
Notice that if $1 - b_1 + b_2^+$ is odd for a four-manifold $M$, then $M$ does not admit any generalized complex structure as it is not even almost complex by \cite{HH58} or \cite[Theorem 1.4.13]{GS99}. The requirement that the Euler characteristic is positive, on the other hand, seems to be more of a limitation of our methods.
\begin{proof}
We will first prove that the manifolds in the list above admit elliptic symplectic structures with imaginary parameters  as long as the Euler characteristic is non-negative and then show that these elliptic symplectic structures come from a generalized complex one if $1 - b_1 + b_2^+$ is even.

In Examples~\ref{th:CP2n},~\ref{prop:barcp2} and~\ref{prop:S2xS2} we produced elliptic symplectic structures on $\cc P^2$, $\bar{\cc P^2}$ and $S^2\times S^2$ with $3$, $3$ and $4$ points in $D(2)$ respectively. These are all locally complex with imaginary parameter, so that by applying Theorem~\ref{thm:ellipticglueing} inductively we obtain elliptic symplectic structures on
 $X_{n,0}$ and $\hat X_{n,m,0}$ for all values of $n$, $m$, including the case $n = m =0$ by Example~\ref{prop:S4}. The number of points in $D(2)$ in these manifolds is, respectively, $n+m+2$ and $2n+ 2$. By Corollary~\ref{prop:handleattachments} we can self-connect sum these spaces up to $\floor{\frac{n+m+2}{2}}$  and $n+1$ times, respectively, to obtain elliptic symplectic structures with imaginary parameters on the spaces of the list with non-negative Euler characteristic.

To prove that the elliptic symplectic structures constructed above are induced by stable generalized complex structures, it suffices to show that the parity is $1$, by Lemma~\ref{lem:paritycomplexlog}. Due to Theorem~\ref{thm:ellipticglueing} and Corollary~\ref{prop:handleattachments}, the parity of the symplectic structure for both families is $(-1)^{n-1+\ell}$, which is positive if and only if $n-1 +\ell$ is even, that is, $1 - b_1 + b_2^+$ is even.
\end{proof}
\begin{remark}
Several of the manifolds in Theorem \ref{th:examples}, although not all, have already appeared before. The family $\hat X_{n,m,0}$ with $n > 0$ and $m\geq 0$ can be found in \cite{CG09}. The manifolds $X_{n,0}$, $X_{n,1}$ and $\hat X_{n,m,1})$ appeared in \cite{Tor12}. The examples with $\ell>1$ have not appeared before. The main advantage of our approach is that it is direct. We construct geometric structures on manifolds which are connected sums by showing that the connected sum operation is compatible with the structures in question. This is in contrast with those references, which instead construct manifolds with the desired structure via surgeries and then determine the resulting diffeomorphism type at a later stage.
\end{remark}
\begin{remark}\label{rem:not complex}
The manifolds in both families in Theorem \ref{th:examples} do not admit complex or symplectic structures if $n >1$. Indeed, for $n>1$, the manifolds are connected sums of manifolds with $b_2^+>0$ and by results from Seiberg--Witten theory due to Taubes (see \cite{T94}) no such manifold admits a symplectic structure.

To prove the non-existence of complex structures requires a slightly longer argument. For $\ell$ even, if the manifolds admitted complex structures they would be K\"ahler and hence also symplectic, but we ruled out this possibility already. The argument for $\ell$ odd comes from a paper by Belgun \cite{MR1760667} and goes as follows. It follows from the Kodaira classification of surfaces that if one of these manifolds were complex, call it $X$, then its Kodaira dimension would be $1$ and $X$ would be an elliptic surface, possibly with multiple fibers, $X \to B$. But in this case there is a finite cover, $\tilde X$ of $X$, corresponding to a branched cover $\tilde B$ of $B$ which is a genuine fibration: $\tilde X \to \tilde B$. By a result of Mehara \cite{MR578873}, any such $\tilde X$ is a quotient of $\cc^2$ by a discrete group. In particular, we conclude that $\tilde X$ and (hence also $X$) would be aspherical, which is not the case for our manifolds. We are thankful to Ornea and Vuletescu for pointing us towards this argument.
\end{remark}

\providecommand{\bysame}{\leavevmode\hbox to3em{\hrulefill}\thinspace}
\providecommand{\MR}{\relax\ifhmode\unskip\space\fi MR }
\providecommand{\MRhref}[2]{%
  \href{http://www.ams.org/mathscinet-getitem?mr=#1}{#2}
}
\providecommand{\href}[2]{#2}

\end{document}